\documentclass[11pt]{article}
\usepackage[english]{babel}
\usepackage{amsmath,amsthm,amssymb}
\usepackage{mathrsfs}
\usepackage{bbm}
\usepackage{amsfonts}
\usepackage[margin=0.8in]{geometry}
\usepackage{graphicx}
\usepackage{caption}
\usepackage{subcaption}
\usepackage{wasysym}
\usepackage{enumerate}
\usepackage{algorithm2e}
\usepackage{tabularx}
\usepackage{color}
\usepackage{wrapfig}
\usepackage{todonotes}

\newtheorem{thm}{Theorem}[section]
\newtheorem{ass}[thm]{Assumption}

\newtheorem{lem}[thm]{Lemma}
\newtheorem{prop}[thm]{Proposition}

\theoremstyle{definition}
\newtheorem{defn}[thm]{Definition}

\theoremstyle{remark}
\newtheorem{rem}[thm]{Remark}

\numberwithin{equation}{section}
\newcommand{\R}{\mathbb R}
\newcommand{\eps}{\varepsilon}

\newcommand{\bbC}{\mathbb C}
\newcommand{\bfC}{\mathbf C}

\newcommand{\bbF}{\mathbb F}
\newcommand{\bbT}{\mathbb T}

\newcommand{\bbN}{\mathbb N}

\newcommand{\mcA}{\mathcal{A}}

\newcommand{\mcB}{\mathcal{B}}
\newcommand{\mcC}{\mathcal C}

\newcommand{\mcE}{\mathcal E}

\newcommand{\mcF}{\mathcal F}

\newcommand{\mcT}{\mathcal T}
\newcommand{\mcP}{\mathcal P}

\newcommand{\mcH}{\mathcal H}
\newcommand{\mcK}{\mathcal K}
\newcommand{\mcU}{\mathcal U}

\newcommand{\mcR}{\mathcal R}
\newcommand{\mcS}{\mathcal S}

\newcommand{\mcV}{\mathcal V}

\newcommand{\mcO}{\mathcal{O}}
\newcommand{\E}{\mathbb{E}}
\newcommand{\Prob}{\mathbb{P}}

\newcommand{\bx}{\mathbf{x}}
\newcommand{\by}{\mathbf{y}}
\newcommand{\esssup}{\mathop{\rm{ess}\sup}}
\newcommand{\essinf}{\mathop{\rm{ess}\inf}}

\newcommand{\argmin}{\mathop{\arg\min}}
\newcommand{\ett}{\mathbbm{1}}

\newcommand{\cadlag}{c\`adl\`ag~}

\newcommand{\Pred}{\mathcal{P}}
\newcommand{\Prog}{{\rm Prog}}

\newcommand{\ie}{\textit{i.e.\ }}
\newcommand{\eg}{\textit{e.g.\ }}

\begin{document}

\title{Feedback Stopping Rules in Path-Dependent Controller–Stopper Games}

\author{Magnus Perninge\footnote{M.\ Perninge is with the Department of Mathematics, M\"alardalen University, V\"aster{\aa}s,
Sweden. e-mail: magnus.perninge@mdu.se.}} %
\maketitle
% ----------------------------------------------------------------
\begin{abstract}
We investigate finite-horizon, zero-sum controller–stopper games in which the stopper observes the state process and implements a feedback stopping rule. Building on a nonlinear Snell envelope representation for a related game established in a companion paper, we prove that our game admits a value by extending the associated first-contact principle. Our approach is purely probabilistic and yields an optimal feedback stopping rule for path-dependent systems while allowing for degeneracy in the underlying stochastic differential equation (SDE).

As an application, we consider nonzero-sum controller-stopper games and show that the optimal feedback stopping rule for a zero-sum game appears as a component of a $\varepsilon$-Nash equilibrium for every $\varepsilon>0$.
\end{abstract}

% ----------------------------------------------------------------
\section{Introduction}
The controller-stopper game is a two-player stochastic differential game that appears naturally in a variety of applications ranging from finance to economics and engineering. The game can arise either as a natural consequence of the strategic interaction between two players, or as the result of parameter uncertainty rendering a robust setup. In the path-dependent framework, the underlying stochastic process is the solution to a controlled non-Markovian stochastic differential equation (SDE),
\begin{align}\label{ekv:fwd-sde}
    X^{t,\bx;u}_s=x(s\wedge t)+\int_t^{s\vee t}a(r,X^{t,\bx;u},u_r)dr+\int_t^{s\vee t}\sigma(r,X^{t,\bx;u},u_r)dW_r,\quad\forall s\in [0,T],
\end{align}
where the continuous trajectory $\bx\in C([0,t]\to\R^d)$ models the history of the process at time $t$. In the zero-sum version of the game, the controller influences the dynamics of $X^{t,\bx,u}$ by choosing a progressively measurable control process $(u_s:t\leq s\leq T)$, taking values in the compact set $U\subset\R^d$, with the objective of minimizing
\begin{align*}
  J(t,\bx;u,\tau)&:=\E\Big[\psi(\tau,X^{t,\bx;u})+\int_t^{\tau}f(s,X^{t,\bx;u},u_s)ds\,\Big|\,\mcF_t\Big].
\end{align*}
The stopper, on the other hand, selects a stopping time $\tau$, with the objective of maximizing the same quantity.

Given its importance, the controller-stopper game has been extensively investigated, see \eg \cite{Karatzas2001,Weerasinghe2006,Zamfirescu08,BayraktarHuang2013,BayraktarYao14,NutzZhang15, Choukroun15, EkstromSalami, Bodnariu2024, EkstromDeFinetti, cont-stop-P1}. Within the Markovian framework, \ie when for all times $s\in [0,T]$ and state trajectories $\bx\in C([0,T]\to\R^d)$, the quadruple $(a(s,\bx),\sigma(s,\bx),\psi(s,\bx),f(s,\bx))$ only depends on $\bx$ through is current value $\bx(s)$, a complete characterization of the game value as the unique viscosity solution to a nonlinear variational inequality was established in \cite{BayraktarHuang2013}, while a corresponding BSDE representation was derived in \cite{Choukroun15}.

The multi-dimensional, path-dependent setting was investigated in \cite{Zamfirescu08,BayraktarYao14,NutzZhang15,cont-stop-P1}. By applying a change of measure technique that effectively alter the quadratic variation of the driving noise process \cite{BayraktarYao14} and \cite{NutzZhang15} both allow models with controlled volatility under a non-degeneracy assumption. In a paper accompanying the present work, \cite{cont-stop-P1}, a different change of measure technique was applied, where the control $u$ is replaced by a Poisson point process
\begin{align*}
  I^t_s=\beta_0+\int_t^s\int_U(e-I^t_{r-})\mu(dr,de),\quad\forall s\in [t,T]
\end{align*}
and the controller chooses the distribution of $I^t$ by altering the compensator of $\mu$. The introduction of an auxiliary process eliminates the need for a non-degeneracy assumption and an important feature of \cite{cont-stop-P1} is that it allows for situations where the volatility matrix $\sigma^\top\sigma(s,\bx)$ becomes singular.

Specifically, the accompanying paper \cite{cont-stop-P1} derives a nonlinear Snell envelope representation for the game value
\begin{align}\label{ekv:value-strat}
  v(t,\bx)=\essinf_{u\in\mcU_t}\esssup_{\tau\in\mcT_t}J(t,\bx;u,\tau)=\esssup_{\tau^S\in\mcT^S_t}\essinf_{u\in\mcU_t}J(t,\bx;u,\tau^S(u)),
\end{align}
where $\mcU_t$ is the set of admissible controls, $\mcT_t$ is the set of stopping times $\tau\geq t$ and $\mcT^S_t$ is the set of non-anticipative maps $\tau^S:\mcU_t\to\mcT_t$.

As pointed out in several works, implementing a strategy from $\mcT^S_t$, would require either that the stopper observes the opponent’s control in real time, which is unrealistic in most applications, or that the controller reveals their intended control actions to the stopper, which is incompatible with the non-cooperative nature of the game. In the present work, we extend the results from \cite{cont-stop-P1} by showing that the value in \eqref{ekv:value-strat} coincides with the value of a controller stopper game on the same cost/reward functional, where the stopper is only allowed to implement feedback stoping rules. That is,
\begin{align}\label{ekv:value-feedback}
  v(t,\bx)=\essinf_{u\in\mcU_t}\esssup_{\tau^F\in\mcT^F_t}J(t,\bx;u,\tau^F(X^{t,\bx;u})) = \esssup_{\tau^F\in\mcT^F_t}\essinf_{u\in\mcU_t}J(t,\bx;u,\tau^F(X^{t,\bx;u})),
\end{align}
where $\mcT^F_t$ is the set of non-anticipative maps $\tau^F:C([0,T]\to\R^d)\to [t,T]$.

Our approach is based on extending the first contact principle to controller-stopper games by proving that the feedback stopping rule
\begin{align}\label{ekv:opt-stop}
  \tau^{F,*}_t:\bx\mapsto \inf\{s\geq t:v(s,\bx)=\psi(s,\bx)\}\in\mcT^F_t
\end{align}
induces an optimal non-anticipative strategy in \eqref{ekv:value-strat}.

As an application of our first contact principle we consider nonzero-sum games where the objective of the controller is to maximize
\begin{align*}
  J^C(t,\bx;u,\tau)&:=\E\Big[\psi^C(\tau,X^{t,\bx;u})+\int_0^{\tau}f^C(s,X^{t,\bx;u},u_s)ds\Big].
\end{align*}
We prove that, in addition to being optimal for the zero-sum game, we can for each $\eps>0$, find a corresponding control $u^\eps\in\mcU_t$ such that
\begin{align}\label{ekv:eps-Nash}
\begin{cases}
  J^{C}(t,\bx;u^\eps;\tau^{F,*}(X^{t,\bx;u^\eps}))\geq J^{C}(t,\bx;u;\tau^{F,*}(X^{t,\bx;u}))-\eps,
  \\
  J(t,\bx;u^\eps;\tau^{F,*}(X^{t,\bx;u^\eps}))\geq J(t,\bx;u^\eps;\tau^F(X^{t,\bx;u^\eps}))-\eps,
\end{cases}
\end{align}
for all $u\in\mcU$ and $\tau^F\in\mcT^F$. The pair $(u^\eps,\tau^{F,*})$ is referred to as an $\eps$-Nash-equilibrium and approximates an equilibrium in the sense that neither player can improve its payoff by more than $\eps$ through unilateral deviation.

\medskip

\textbf{Outline} The next section provides preliminary definitions, states the main assumptions used throughout the paper, and recalls several prior results that are needed in the subsequent analysis.

Section~\ref{sec:main-res} formally states the main result of the paper, summarized in Theorem~\ref{thm:tau-F-robust}. Moreover, it is shown that the theorem is equivalent to a duality result involving the optimal control problem terminated at $\tau^{F,*}(X^{u})$. Section~\ref{sec:opt-stop-zs} provides a detailed proof of the corresponding duality result thereby proving Theorem~\ref{thm:tau-F-robust}.

In Section~\ref{sec:nz}, we apply our results to the nonzero-sum version of the controller-stopper game and prove that the optimal feedback stopping rule for the zero-sum game forms a component of $\eps$-Nash equilibria.\\

\section{Preliminaries}
\subsection{Probabilistic setup}
To save notation, rather than working with two separate probability spaces, we follow the convention of \cite{cont-stop-P1} and use the same probability space as a basis for both the primal and the dual formulations of the controller-stopper game. We thus let $(\Omega,\mcF,\Prob)$ be a complete probability space supporting a $d$-dimensional Brownian motion, denoted by $W$, and an independent Poisson random measure $\mu$  on $[0,T]\times U$ with compensator $dt\otimes \lambda(de)$, where $\lambda$ is a finite measure on $U$ with full topological support. We denote by $\bbF:=(\mcF_t)_{t\geq 0}$ the augmented natural filtration generated by $W$, while $\bbF^{\mcR}:=(\mcF^{\mcR}_t)_{t\geq 0}$ is the augmented natural filtration generated by both $W$ and $\mu$.

In this setting, for each $E\in\mathcal{B}(U)$, the compensated process
\[
\tilde{\mu}([0,t],E) := \mu([0,t],E) - t\lambda(E), \quad t\geq 0,
\]
is an $\mathbb{F}^{\mathcal R}$-martingale.

\subsection{Notations}

\noindent Throughout, we use the following notation, where $T> 0$ is the maximal duration of the game:
\begin{itemize}
  \item We denote by $\bfC^d$ the set of continuous functions $\bx:[0,T]\to\R^d$ equipped with the supremum norm $\|\cdot\|_T$, where $\|\bx\|_t:=\sup_{s\in[0,t]}|\bx(s)|$.
  \item We let $\bbC:=(\mcC_t)_{t\in [0,T]}$ be the filtration generated by the coordinate map, $\mcC_t:=\sigma(\bx\mapsto\bx(s):s\in [0,t])$ on $\bfC^d$.
  \item For any two maps, $\bx,\by:[0,T]\to\R^d$, we define concatenation at $t\in [0,T]$ as $(\bx \otimes_t \by)(s):=\ett_{[0,t]}(s)\bx(s) + \ett_{(t,T]}(s)\by(s)$ for all $s\in [0,T]$.
  \item We define the set $\Lambda:=[0,T]\times \bfC^d$ which we equip with the pseudo-metric
      \begin{align}
        \mathbf d_\Lambda[(t,\bx),(t',\bx')]:=|t'-t|+\|\bx'(\cdot\wedge t')-\bx(\cdot\wedge t)\|_T.
      \end{align}
  \item For a measure space $(\tilde\Omega,\tilde\mcF)$ and a filtration $\tilde\bbF$ on $\tilde\mcF$ we let $\Prog(\tilde\bbF)$ (resp. $\Pred(\tilde\bbF)$) denote the $\sigma$-algebra of $\tilde\bbF$-progressively (resp. $\tilde\bbF$-predictably) measurable subsets of $\R_+\times \tilde\Omega$.
  \item We let $\mcU_t$ be the set of $\Prog(\bbF)$-measurable processes $(u_s:t\leq s\leq T)$ valued in the compact set $U\subset\R^d$.
  \item We let $\mcT^F$ be the set of maps $\tau^F:\bfC^d\to [0,T]$ such that $\{\bx\in\bfC^d:\tau^F(\bx)\le t\}\in\mcC_t$ for all $t\in[0,T]$. For each $t\in[0,T]$, we let $\mcT_t^F$ denote the subset of $\mcT^F$ consisting of maps $\tau^F$ satisfying $\tau^F(\bx)\ge t$ for all $\bx\in\bfC^d$.
  \item We let $\mcT$ be the set of all $[0,T]$-valued $\bbF$-stopping times and for each $\eta\in\mcT$, we let $\mcT_\eta$ be the corresponding subset of stopping times $\tau$ such that $\tau\geq \eta$, $\Prob$-a.s.
  \item Similarly, we let $\mcT^\mcR$ be the set of all $[0,T]$-valued $\bbF^{\mcR}$-stopping times and for each $\eta\in\mcT^\mcR$, we let $\mcT^\mcR_\eta$ be the corresponding subset of stopping times $\tau$ such that $\tau\geq \eta$, $\Prob$-a.s.
  \item We let $\mcV$ be the set of all $\Pred(\bbF^\mcR)\otimes\mcB(U)$-measurable bounded maps $\nu:[0,T]\times\Omega\times U\to [0,\infty),\,(t,\omega,e)\mapsto\nu_t(\omega,e)$. Moreover, for $n\in\bbN$, we let $\mcV^n$ denote the subset of maps $\nu:[0,T]\times\Omega\times U\to [0,n]$.
  \item For $p\geq 1$, $t\in [0,T]$ and $\tau\in\mcT^\mcR_t$, we let $\mcS^{p}_{t,\tau}$ be the set of all $\R$-valued, $\Prog(\bbF^{\mcR})$-measurable \cadlag processes $(Z_s: s\in[t,\tau])$ such that $\|Z\|_{\mcS^{p,\tau}}:=\E\Big[\sup_{s\in [t,\tau]} |Z_s|^p\Big]^{1/p}<\infty$.
      When $\tau=T$, we use the shorter notation $\mcS^{p}_t$.
  \item We let $\mcA^{p}_{t,\tau}$ denote the subset of $\mcS^{p}_{t,\tau}$ consisting of all $\mcP(\bbF^{\mcR})$-measurable, nondecreasing processes $Z$ satisfying $Z_t=0$. Moreover, we let $\mcA^p_t:=\mcA^{p}_{t,T}$.
  \item We let $\mcH^{p}_{t,\tau}(W)$ denote the set of all $\R^d$-valued, $\mcP(\bbF^{\mcR})$-measurable processes $(Z_s: s\in[t,\tau])$ such that $\|Z\|_{\mcH^{p}_{t,\tau}(W)}:=\E\left[(\int_t^\tau |Z_s|^2 ds)^{p/2}\right]^{1/p}<\infty$. When $\tau=T$, we use the notation $\mcH^{p}_t(W)$.
  \item We let $\mcH^{p}_{t,\tau}(\mu)$ denote the set of all $\R$-valued, $\mcP(\bbF^{\mcR})\otimes \mcB(U)$-measurable mappings $(Z_s(e): s\in[t,\tau],e\in U)$ such that $\|Z\|_{\mcH^{p}_{t,\tau}(\mu)}:=\E\Big[\int_t^\tau\!\!\int_U |Z_s(e)|^p\lambda(de)ds\Big]^{1/p}<\infty$ and set $\mcH^{p}_t(\mu):=\mcH^{p}_{t,T}(\mu)$.
\end{itemize}

Unless otherwise stated, all inequalities involving random variables are assumed to hold $\Prob$-a.s.

\subsection{Assumptions}
We assume that the coefficients of the forward SDE satisfy the following conditions:
\begin{ass}\label{ass:onSDE}
\begin{enumerate}[i)]
  \item\label{ass:onSDE-a-sigma} The coefficients $a:[0,T]\times\bfC^d\times U\to\R^{d}$ and $\sigma:[0,T]\times\bfC^d\times U\to\R^{d\times d}$ have components that are $\Prog(\bbC)\otimes\mcB(U)$-measurable, continuous in $u$, uniformly on sets of the form $\{(t,\bx): \|\bx\|_t \le K\}$ for each $K>0$, satisfy the linear growth condition
  \begin{align}\label{ekv:a-sigma-growth}
    |a(t,\bx,u)|+|\sigma(t,\bx,u)|&\leq C(1+\|\bx\|_t)
  \end{align}
  and the Lipschitz continuity
  \begin{align*}
    |a(t,\bx,u)-a(t,\bx',u)|+|\sigma(t,\bx,u)-\sigma(t,\bx',u)|&\leq k_{a,\sigma}\|\bx'-\bx\|_t
  \end{align*}
  for all $(t,\bx,\bx')\in [0,T]\times \bfC^d\times\bfC^d$ and $u\in U$.
\end{enumerate}
\end{ass}

Moreover, we make the following assumptions on the coefficients in the cost/reward functional $J$:

\begin{ass}\label{ass:oncoeff}
There are constants $C>0$ and $q>0$ in addition to a family of moduli of continuity $(\varpi_K)_{K\geq 0}$ such that:
\begin{enumerate}[i)]
  \item\label{ass:oncoeff-f} The running cost/reward $f:[0,T]\times \bfC^d\times U\to\R$ is $\Prog(\bbC)\otimes\mcB(U)$-measurable and satisfies the growth condition
  \begin{align*}
    |f(t,\bx,u)|\leq C(1+\|\bx\|_t^q).
  \end{align*}
  Moreover, for each $K>0$,
  \begin{align*}
    |f(t,\bx',u')-f(t,\bx,u)|\leq \varpi_K(\|\bx'-\bx\|_t+|u'-u|)
  \end{align*}
  for all $(t,\bx,\bx',u,u')\in [0,T]\times \bfC^d\times\bfC^d\times U\times U$, with $\|\bx\|_t\vee \|\bx'\|_t\leq K$.
  \item The terminal reward $\psi:[0,T]\times\bfC^d\to\R$ is $\Prog(\bbC)$-measurable with \cadlag trajectories $t\mapsto \psi(t,\bx)$ for all $\bx\in\bfC^d$, and satisfies a polynomial growth condition, \ie
  \begin{align*}
    |\psi(t,\bx)|\leq C(1+\|\bx\|_t^q).
  \end{align*}
  Moreover, for every $K>0$,
  \begin{align*}
    \psi(t',\bx')-\psi(t,\bx)\leq \varpi_K(\mathbf d_\Lambda[(t,\bx),(t',\bx')]),
  \end{align*}
  whenever $0\leq t\leq t'\leq T$ and $\bx,\bx'\in\bfC^d$ satisfy $\|\bx\|_t\vee \|\bx'\|_{t'}\leq K$.
\end{enumerate}
\end{ass}

\subsection{Preliminary estimates}
We recall some preliminary estimates from \cite{cont-stop-P1}, where the pseudo-metric $\rho_t:\mcU_t\times\mcU_t\to\R_+$ on $\mcU_t$ is defined as
\begin{align*}
  \rho_t(u,\tilde u):=\E\Big[\int_{t}^T|u_s-\tilde u_s| ds\Big].
\end{align*}

\begin{prop}\label{prop:Xu-stab}
  For any $p\geq 1$, there is a $C_p>0$ such that
  \begin{align}\label{ekv:Xu-growth}
    \E\Big[\sup_{s\in[t,T]}|X^{t,\bx;u}_s|^p\,\Big|\,\mcF_t\Big]\leq C_p(1+\|\bx\|^p_t),\quad\Prob-\text{a.s.}
  \end{align}
  for all $u\in\mcU_t$. Moreover, there is a $C>0$ such that for any $(t,\bx),(\tilde t,\tilde \bx)\in\Lambda$ and any sequences $(u^i)_{i\in\bbN}\subset\mcU_t$ and $(\tilde u^i)_{i\in\bbN}\subset\mcU_{\tilde t}$ such that $\rho_{t\vee\tilde t}(u^i,\tilde u^i)\to 0$, we have
  \begin{align}\label{ekv:Xu-stab}
  \limsup_{i\to\infty}\E\Big[\sup_{s\in [0,T]}|X^{t,\bx;u^i}_s- X^{\tilde t,\tilde\bx;\tilde u^i}_s|^2\Big]\leq C(|t-\tilde t|(1+\|\bx\|^2_t+\|\tilde\bx\|^2_{\tilde t})+\|\bx(\cdot\wedge t)-\tilde\bx(\cdot\wedge \tilde t)\|^2_T).
\end{align}
\end{prop}

\begin{proof}
A proof based on standard arguments can be found in \cite{cont-stop-P1} (see Proposition 2.5 therein).
\end{proof}

Under the regularity assumptions imposed on $\psi$ and $f$, the above lemma yields the following continuity property of the cost/reward functional:

\begin{lem}\label{lem:J-cont}
  For any sequences
  \begin{itemize}
    \item $(t_i,\bx_i)_{i\in\bbN}\subset[0,T]\times\bfC^d$ and $(\tilde t_i,\tilde \bx_i)_{i\in\bbN}\subset[0,T]\times\bfC^d$ that are bounded with $\textbf d_\Lambda[(t_i,\bx_i),(\tilde t_i,\tilde\bx_i)]\to 0$,
    \item $(u^i)_{i\in\bbN}$ and $(\tilde u^i)_{i\in\bbN}$, with $u^i\in\mcU_{t_i}$, $\tilde u^i\in\mcU_{\tilde t_i}$ and $\rho_{t_i\vee\tilde t_i}(u^i,\tilde u^i)\to 0$; and
    \item $(\tau_i)_{i\in\bbN}$ and $(\tilde \tau_i)_{i\in\bbN}$ with $\tau_i\in\mcT_{t_i}$, $\tilde\tau_i\in\mcT_{\tilde t_i}$ $\tau_i\leq \tilde \tau_i$ and $\tilde\tau_i-\tau_i\to 0$, $\Prob$-a.s.~as $i\to\infty$,
  \end{itemize}
  we have
  \begin{align}\label{ekv:seq-to-0}
    \limsup_{i\to\infty}\E\Big[\psi(\tau_i,X^{t_i,\bx_i;u^i})-\psi(\tilde\tau_i,X^{\tilde t_i,\tilde\bx_i;\tilde u^i}) +\int_{t_i}^{\tau}f(s,X^{t_i,\bx_i;u^i},u^i_s)ds-\int_{\tilde t_i}^{\tilde\tau_i}f(s,X^{\tilde t_i,\tilde \bx_i;\tilde u^i},\tilde u^i_s)ds\Big]\leq 0.
  \end{align}
\end{lem}

\begin{proof}
This corresponds to Lemma 4.4 in \cite{cont-stop-P1}.
\end{proof}

\subsection{A dual representation of the zero-sum game}
In the companion paper \cite{cont-stop-P1}, the value in \eqref{ekv:value-strat} is represented through a dual randomized formulation, in which the control is replaced by an auxiliary Poisson point process. In this regard, we introduce the uncontrolled state pair $(I^t,X^{t,\bx})$ that satisfy the forward SDE
\begin{align*}
  \begin{cases}
    I^t_s=\beta_0+\int_t^s\!\!\int_U(e-I^t_{r-})\mu(dr,de),\quad\forall s\in [t,T],\\
    X^{t,\bx}_s=\bx(s\wedge t)+\int_t^{s\vee t} a(r,X^{t,\bx},I^t_r)dr+\int_t^{s\vee t}\sigma(r,X^{t,\bx},I^t_r)dW_r,\quad\forall s\in [0,T],
  \end{cases}
\end{align*}
where $\beta_0 \in U$ is the initial value\footnote{In \cite{cont-stop-P1} it was shown that the value of the dual game, $v^\mcR$, is independent of the choice of $\beta_0$.} of $I^t$ at time $t$. The corresponding reward/cost functional is defined for each $(t,\bx)\in[0,T]\times \bfC^d$ and $(\nu,\tau)\in \mcV\times\mcT^\mcR_t$, as
\begin{align*}
  J^{\mcR}(t,\bx;\nu,\tau)&:=\E^\nu\Big[\psi(\tau,X^{t,\bx})+\int_t^{\tau}f(s,X^{t,\bx},I^t_s)ds\,\Big|\,\mcF_t\Big],
\end{align*}
where $\E^\nu$ is expectation with respect to the probability measure $\Prob^\nu$ on $(\Omega,\mcF)$ defined by $d\Prob^\nu:=\kappa^\nu_T d\Prob$, with
\begin{align*}
\kappa^{\nu}_s&:=\mcE_s\Big(\int_{t}^\cdot\!\!\int_U(\nu_r(e)-1)(\mu(dr,de)-\lambda(de)dr)\Big)
\\
&:=\exp\Big(\int_{t}^s\!\!\!\int_U(1-\nu_r(e))\lambda(de)dr\Big)\prod_{t<\sigma_j\leq s}\nu_{\sigma_j}(\zeta_j)
\end{align*}
and $(\sigma_j,\zeta_j)_{j\in\bbN}$ are the consecutive jump times and corresponding marks of $\mu$.

We are now ready to define the randomized version of the controller-stopper game as
\begin{align}\label{ekv:dual-game}
  v^\mcR(t,\bx)=\essinf_{\nu\in\mcV}\esssup_{\tau\in\mcT^\mcR_t}J^{\mcR}(t,\bx;\nu,\tau).
\end{align}

The main result in the accompanying paper \cite{cont-stop-P1} is the following
\begin{thm}\label{thm:zs-game}
  There exists a deterministic, $\bbC$-progressively measurable, continuous map $v^\mcR:[0,T]\times\bfC^d\to\R$ that satisfies \eqref{ekv:dual-game} and in addition
  \begin{align}\label{ekv:value-strat-rep}
    v^\mcR(t,\bx)=\essinf_{u\in\mcU_t}\esssup_{\tau\in\mcT_t}J(t,\bx;u,\tau)=\esssup_{\tau^S\in\mcT^S_t}\essinf_{u\in\mcU_t}J(t,\bx;u,\tau^S(u)).
  \end{align}
\end{thm}
In particular, \eqref{ekv:value-strat-rep} implies that $v^\mcR\equiv v$, where $v$ is defined in \eqref{ekv:value-strat}.

\subsection{A corresponding BSDE}
The dual game with value function $v^\mcR$ was thoroughly analyzed in \cite{imp-stop-game}, where a representation in terms of a nonlinear Snell envelope was derived. We recall some key results from \cite{imp-stop-game}.

Consider the sequence $(Y^{t,\bx,n},Z^{t,\bx,n},V^{t,\bx,n},K^{+,t,\bx,n})\in\mcS^2_{t,\tau}\times\mcH^2_{t}(W)\times \mcH^2_{t}(\mu)\times \mcK^2_{t}$ defined for each $n\in\bbN$ as the unique solution to the reflected BSDE that penalizes negative jumps,
\begin{align}\label{ekv:rbsde-pen}
\begin{cases}
  Y^{t,\bx,n}_s=\psi(T,X^{t,\bx})+\int_s^T f(r,X^{t,\bx},I^t_r)dr-n\int_s^T\!\!\int_U(V^{t,\bx,n}_r(e))^-\lambda(de)dr-\int_s^T Z^{t,\bx,n}_r dW_r
  \\
  \quad-\int_s^T\!\!\int_U V^{t,\bx,n}_r(e)\mu(dr,de)+K^{+,t,\bx,n}_T-K^{+,t,\bx,n}_s,\quad\forall s\in [t,T]
  \\
  Y^{t,\bx,n}_s\geq \psi(s,X^{t,\bx}),\:\forall s\in [t,T]\quad\text{and}\quad \int_t^T(Y^{t,\bx,n}_s - \psi(s,X^{t,\bx}))d K^{+,t,\bx,n}_s=0.
\end{cases}
\end{align}
It is well known (see \eg \cite{QuenSul14}) that $\tau_n:=\inf\{s\geq t:Y^{t,\bx,n}_s=\psi(s,X^{t,\bx})\}\in\mcT^\mcR_t$ is an optimal stopping time for the corresponding optimal stopping problem.

Now, the sequence of processes $(Y^{t,\bx,n})_{n\in\bbN}$ is non-increasing and bounded from below by the process $\psi(\cdot,X^{t,\bx})$, implying the existence of a $\Prog(\bbF^{\mcR})$-measurable process $Y^{t,\bx}$ such that $Y^{t,\bx,n}\searrow Y^{t,\bx}$ pointwise, $\Prob$-a.s.~as $n\to\infty$. In \cite{imp-stop-game} it was shown that $Y^{t,\bx}\in\mcS^2_t$ satisfies
\begin{align}\label{ekv:zs-nonlin-Snell}
  Y^{t,\bx}_s=\esssup_{\tau\in\mcT^\mcR_t}Y^{t,\bx;\tau}_s,
\end{align}
where for each $\tau\in\mcT^\mcR_t$, the process $Y^{t,\bx;\tau}$ is the first component in the quadruple of processes\\ $(Y^{t,\bx;\tau},Z^{t,\bx;\tau},V^{t,\bx;\tau},K^{-,t,\bx;\tau})\in\mcS^2_{t,\tau}\times\mcH^2_{t,\tau}(W)\times \mcH^2_{t,\tau}(\mu)\times \mcK^2_{t,\tau}$ which constitutes the unique maximal solution to the BSDE with constrained jumps
  \begin{align}\label{ekv:bsde-c-jmp}
    \begin{cases}
      Y^{t,\bx;\tau}_s=\psi(\tau,X^{t,\bx})+\int_s^\tau f(r,X^{t,\bx},I^t_r)dr-\int_s^\tau Z^{t,\bx;\tau}_r dW_r-\int_s^\tau\!\!\int_U V^{t,\bx;\tau}_r(e)\mu(dr,de)
      \\
      \quad-( K^{-,t,\bx;\tau}_\tau-K^{-,t,\bx;\tau}_s),\quad\forall s\in [t,\tau]
      \\
      V^{t,\bx;\tau}_s(e)\geq 0,\quad d\Prob\otimes ds\otimes \lambda(de)-\text{a.e.}
    \end{cases}
  \end{align}
and that the stopping time
\begin{align}\label{ekv:opt-stop-dual}
  \tau^{\mcR}:=\inf\{r\geq t: Y^{t,\bx}_r=\psi(r,X^{t,\bx})\}
\end{align}
is optimal in the sense that $Y^{t,\bx}_s=Y^{t,\bx;\tau^{\mcR}}_s$ for all $s\in [t,\tau^\mcR]$.

Moreover, Lemma 3.2 in \cite{cont-stop-P1} establishes that
\begin{align}\label{ekv:trunc-saddle}
  Y^{t,\bx,n}_t = \essinf_{\nu\in\mcV^n}\esssup_{\tau\in\mcT^\mcR_t}J^{\mcR}(t,\bx;\nu,\tau)= \esssup_{\tau\in\mcT^\mcR_t}\essinf_{\nu\in\mcV^n}J^{\mcR}(t,\bx;\nu,\tau)
\end{align}
and passing to the limit as $n\to\infty$, we find that
\begin{align*}
  Y^{t,\bx}_t = \essinf_{\nu\in\mcV}\esssup_{\tau\in\mcT^\mcR_t}J^{\mcR}(t,\bx;\nu,\tau) = v^\mcR(t,\bx),
\end{align*}
establishing a representation of the game value $v$ in terms of the nonlinear Snell envelope $Y$.

%The remainder of the paper builds on two key ingredients from the companion paper: the dual representation in Theorem~\ref{thm:zs-game} and the nonlinear Snell envelope characterization of $v$.

\section{Controller-stopper games with feedback stopping rules\label{sec:main-res}}

The main result of the present work is summarized in the following theorem, where we recall the definition of $\tau^{F,*}_t$ as the first hitting time of the value process to the barrier, \ie $\tau^{F,*}_t:\bx\mapsto \inf\{s\geq t:v(s,\bx)=\psi(s,\bx)\}$.

\begin{thm}\label{thm:tau-F-robust}
Allowing only feedback stopping rules does not change the value of the controller-stopper game, \ie
\begin{align}\label{ekv:feedback-game-value}
  v(t,\bx)
  =\esssup_{\tau^F\in\mcT^F_t}\essinf_{u\in\mcU_t}J(t,\bx;u,\tau^{F}(X^{t,\bx;u}))
  =\essinf_{u\in\mcU_t}\esssup_{\tau^F\in\mcT^F_t}J(t,\bx;u,\tau^{F}(X^{t,\bx;u})).
\end{align}
Moreover, the supremum over feedback stopping rules in \eqref{ekv:feedback-game-value} is attained by $\tau^{F,*}_t\in\mcT^F_t$  in the sense that for any $(t,\bx)\in[0,T]\times\bfC^d$,
\begin{align}
  \essinf_{u\in\mcU_t}J(t,\bx;u,\tau^F(X^{t,\bx;u})) \leq \essinf_{u\in\mcU_t}J(t,\bx;u,\tau^{F,*}_t(X^{t,\bx;u}))
\label{ekv:F-star-opt}
\end{align}
for all $\tau^F\in\mcT^F_t$.
\end{thm}

We let $v^{F,*}(t,\bx)$ denote the right-hand side in \eqref{ekv:F-star-opt}, so that
\begin{align}\label{ekv:v-F-star-def}
  v^{F,*}(t,\bx)=\essinf_{u\in\mcU_t}J^{F,*}(t,\bx;u),
\end{align}
where $J^{F,*}(t,\bx;u):=J(t,\bx;u,\tau^{F,*}_t(X^{t,\bx;u}))$.

Our approach to prove the above theorem uses a dual characterization of the optimal control problem in \eqref{ekv:v-F-star-def}, and we define the map $v^{\mcR,F,*}:[0,T]\times\bfC^d\to\R$ as
\begin{align}\label{ekv:dual-stopped-OC}
  v^{\mcR,F,*}(t,\bx):=\essinf_{\nu\in\mcV}J^{\mcR,F,*}(t,\bx;\nu),
\end{align}
where $J^{\mcR,F,*}(t,\bx;\nu):=J^{\mcR}(t,\bx;\nu,\tau^{F,*}_t(X^{t,\bx}))$. We then have the following result, the proof of which is deferred to the next section:
\begin{prop}\label{prop:dual-CP}
The value function for the optimal control problem with horizon $\tau^{F,*}_t(X^{t,\bx;u})$ admits the following dual representation
\begin{align}\label{ekv:new-duality}
  v^{F,*}(t,\bx)= v^{\mcR,F,*}(t,\bx),\quad\forall (t,\bx)\in[0,T]\times\bfC^d.
\end{align}
\end{prop}

Combining the dual characterization in \eqref{ekv:new-duality} with the results of the accompanying paper \cite{cont-stop-P1}, detailed in the previous section, allows us to prove Theorem~\ref{thm:tau-F-robust} in a fairly straightforward manner. First, the fact that $v\equiv v^\mcR$ yields the following lemma.
\begin{lem}\label{lem:equiv}
  $v^{\mcR,F,*}\equiv v$
\end{lem}

\begin{proof}
Observe that the stopping time $\tau^{F,*}_t(X^{t,\bx})$ satisfies $\tau^{F,*}_t(X^{t,\bx})=\inf\{s\geq t:Y^{t,\bx}_s=\psi(s,X^{t,\bx})\}$, where $Y^{t,\bx}$ is the nonlinear Snell envelope defined by \eqref{ekv:zs-nonlin-Snell}. By Theorem 3.1 in \cite{imp-stop-game} this stopping time is optimal and we conclude that
\begin{align*}
  v(t,\bx)=Y^{t,\bx}_t=Y^{t,\bx;\tau^{F,*}_t(X^{t,\bx})}_t=Y^{t,\bx;F,*}_t=v^{\mcR,F,*}(t,\bx)
\end{align*}
for all $(t,\bx)\in[0,T]\times\bfC^d$.
\end{proof}

The proof of Theorem~\ref{thm:tau-F-robust} is now a direct consequence of the duality in Proposition~\ref{prop:dual-CP}.

\begin{proof}[Proof of Theorem~\ref{thm:tau-F-robust}]
It is immediate from the definitions of the sets $\mcT^F_t$ and $\mcT^S_t$ and Theorem~\ref{thm:zs-game} that $v^{F,*}(t,\bx)\leq v(t,\bx)$. Clearly, the natural filtration generated by $X^{t,\bx;u}$ and augmented with all $\Prob$-null sets, denoted $\bbF^{X^{t,\bx;u}}$, is contained in the augmented natural filtration generated by $W$ and $u$. Hence,
\begin{align*}
  \{u\mapsto\tau^F(X^{t,\bx;u}):\tau^F\in\mcT^F_t\}\subset \mcT^S_t.
\end{align*}
In particular, $u\mapsto \tau^{F,*}_t(X^{t,\bx;u}):\mcU_t\to\mcT_t$ belongs to $\mcT^S_t$ and thus
\begin{align*}
  v(t,\bx)&=\esssup_{\tau^S\in\mcT^S_t}\essinf_{u\in\mcU_t}J(t,\bx;u,\tau^{S}(u))
  \\
  &\geq \esssup_{\tau^F\in\mcT^F_t}\essinf_{u\in\mcU_t}J(t,\bx;u,\tau^{F}(X^{t,\bx;u}))
  \\
  &\geq \essinf_{u\in\mcU_t}J(t,\bx;u,\tau^{F,*}_t(X^{t,\bx;u}))=v^{F,*}(t,\bx).
\end{align*}

Combining Lemma~\ref{lem:equiv} and Proposition~\ref{prop:dual-CP} yields that $v=v^{F,*}$ and we conclude that
\begin{align*}
  \esssup_{\tau^F\in\mcT^F_t}\essinf_{u\in\mcU_t}J(t,\bx;u,\tau^{F}_t(X^{t,\bx;u}))=\essinf_{u\in\mcU_t}J(t,\bx;u,\tau^{F,*}_t(X^{t,\bx;u})).
\end{align*}
Moreover, we get that
\begin{align*}
  v(t,\bx)&=\essinf_{u\in\mcU_t}J(t,\bx;u,\tau^{F,*}_t(X^{t,\bx;u}))
  \\
  &\leq\essinf_{u\in\mcU_t}\esssup_{\tau^F\in\mcT^F_t}J(t,\bx;u,\tau^{F}_t(X^{t,\bx;u}))
  \\
  &= \essinf_{u\in\mcU_t}\esssup_{\tau\in\mcT_t}J(t,\bx;u,\tau)
  \\
  &=v(t,\bx),
\end{align*}
which proves both \eqref{ekv:feedback-game-value} and \eqref{ekv:F-star-opt}.
\end{proof}

\section{Proof of Proposition~\ref{prop:dual-CP}\label{sec:opt-stop-zs}}

Since \eqref{ekv:new-duality} is a duality result for a control problem rather than a differential game, one might expect existing results on control randomization in stochastic optimal control to apply. However, the BSDE representation results of \cite{Fuhrman15} and \cite{Bandini18} concern control problems with fixed terminal times and rely crucially on continuity properties of the terminal reward. In the present setting, the terminal time $\tau^{F,*}_t(X^{t,\bx;u})$ may depend on the entire state path in a highly nontrivial manner, and the map
\begin{align*}
\tilde\bx \mapsto \psi\big(\tau^{F,*}_t(\tilde\bx),\tilde\bx\big):\bfC^d\to\R
\end{align*}
typically fails to be continuous. Indeed, the map $\tilde\bx \mapsto \tau^{F,*}_t(\tilde\bx)$ is, in general, only lower semicontinuous.\\

We overcome this difficulty and establish the duality in three steps:
\begin{enumerate}
  \item We first approximate the feedback stopping rule $\tau^{F,*}_t$ from below by a family $(\tau^{F,\zeta}_t)_{\zeta>0}$. We then justify these approximations by establishing stability results for both the primal and dual (randomized) control problems.
  \item For arbitrary $\varrho>0$, we choose a $\varrho$-optimal control $u^\varrho\in\mcU_t$ and introduce an auxiliary probability space supporting a family of randomized controls that approximate $u^\varrho$ with respect to the metric $\rho$.
  \item Finally, we prove a sandwich result showing that the stopping time obtained by applying the feedback rule $\tau^{F,\zeta}_t$ to the randomized state lies between $\tau^{F,2\zeta}_t(X^{t,\bx;u^\varrho})$ and $\tau^{F,*}_t(X^{t,\bx;u^\varrho})$. Combined with the stability results from the first step, this allows us to pass to the limit and establish the desired duality.
\end{enumerate}
The remainder of this section is organized into three subsections corresponding to these steps.

\subsection{An approximation of $\tau^{F,*}_t$}
We will approximate the stopping rule $\tau^{F,*}$ from the left and for each $\zeta>0$, we define the feedback stopping rule
\begin{align*}
  \tau^{F,\zeta}_t:\bfC^d\to [t,T],\quad
  \bx\mapsto \inf\{s\geq t: v(s,\bx)\leq\psi(s,\bx)+\zeta\}\in\mcT^F_t
\end{align*}
and introduce the corresponding optimal control problem
\begin{align*}
  v^{F,\zeta}(t,\bx)=\essinf_{u\in\mcU_t}J^{F,\zeta}(t,\bx;u),
\end{align*}
with cost functional
\begin{align*}
  J^{F,\zeta}(t,\bx;u):=J(t,\bx;u,\tau^{F,\zeta}(X^{t,\bx;u})).
\end{align*}
Then we have the following basic inequality:
\begin{lem}
  For each $(t,\bx)\in [0,T]\times\bfC^d$, we have
  \begin{align*}
    \limsup_{\zeta\to 0}v^{F,\zeta}(t,\bx) \leq v^{F,*}(t,\bx).
  \end{align*}
\end{lem}

\begin{proof}
  Since $\tau^{F,\zeta}_t\nearrow \tau^{F,*}_t$ pointwise on $\bfC^d$, we can use dominated convergence and left-upper semi-continuity of $\psi$ to conclude that
  \begin{align}\nonumber
    \lim_{\zeta\to 0}J^{F,\zeta}(t,\bx;u)&=\lim_{\zeta\to 0}\E\Big[\psi(\tau^{F,\zeta}(X^{t,\bx;u}),X^{t,\bx;u}) + \int_t^{\tau^{F,\zeta}(X^{t,\bx;u})}f(s,X^{t,\bx},u_s)ds\,\Big|\,\mcF_t\Big]
    \\
    &\leq J^{F,*}(t,\bx;u),\label{ekv:J-zeta-dom-con}
  \end{align}
  $\Prob$-a.s.,~for each $(t,\bx)\in [0,T]\times\bfC^d$ and $u\in\mcU_t$. On the other hand, by the definition of the essential infimum, there is for any $(t,\bx)\in [0,T]\times\bfC^d$ and any $\varrho>0$, a corresponding $u^\varrho\in\mcU_t$ such that
  \begin{align*}
    v^{F,*}(t,\bx)\geq J^{F,*}(t,\bx;u^\varrho)-\varrho.
  \end{align*}
  Hence,
  \begin{align*}
    v^{F,\zeta}(t,\bx) - v^{F,*}(t,\bx)\leq J^{F,\zeta}(t,\bx;u^\varrho)-J^{F,*}(t,\bx;u^\varrho)+\varrho.
  \end{align*}
  Taking the limit as $\zeta\to 0$, we thus find that
  \begin{align*}
    \limsup_{\zeta\to 0}v^{F,\zeta}(t,\bx) \leq v^{F,*}(t,\bx)+\varrho
  \end{align*}
  and the desired result follows since $\varrho>0$ was arbitrary.
\end{proof}

\medskip

\begin{rem}
  In fact, whenever $\psi(\cdot,X^{t,\bx;u})$ has a (necessarily positive) jump at $\tau^{F,*}_t(X^{t,\bx;u})$, continuity of $v$ implies that $\tau^{F,*}_t(X^{t,\bx;u})=\tau^{F,\zeta}(X^{t,\bx;u})$ for sufficiently small $\zeta>0$. Therefore, equality holds in \eqref{ekv:J-zeta-dom-con}.
\end{rem}

\medskip

Similarly, we introduce a version of the dual control problem terminated at $\tau^{F,\zeta}$, defined by
\begin{align*}
  v^{\mcR,F,\zeta}(t,\bx)=\essinf_{\nu\in\mcV}J^{\mcR,F,\zeta}(t,\bx;\nu),
\end{align*}
where
\begin{align*}
  J^{\mcR,F,\zeta}(t,\bx;\nu)&:=\E^\nu\Big[\psi(\tau^{F,\zeta}(X^{t,\bx}),X^{t,\bx}) + \int_t^{\tau^{F,\zeta}(X^{t,\bx})}f(s,X^{t,\bx},I^t_s)ds\,\Big|\,\mcF_t\Big].
\end{align*}
The following lemma provides a stability result ensuring adequacy of the above approximation of the dual game.

\begin{lem}\label{lem:rand-pert-conv}
  For each $(t,\bx)\in [0,T]\times\bfC^d$, we have
  \begin{align*}
    \lim_{\zeta\to 0}v^{\mcR,F,\zeta}(t,\bx) = v^{\mcR,F,*}(t,\bx).
  \end{align*}
\end{lem}

\begin{proof}
Fix $\zeta>0$ and note that $\tau^{F,\zeta}_t(X^{t,\bx})\leq \tau^{F,*}_t(X^{t,\bx})$, $\Prob$-a.s. Hence, on the set $[t,{\tau^{F,\zeta}_t(X^{t,\bx})}]$, the quadruple\footnote{Here, the right hand-side is defined in \eqref{ekv:bsde-c-jmp}.}
\begin{align*}
  (Y^{t,\bx;F,*},Z^{t,\bx;F,*},V^{t,\bx;F,*},K^{-,t,\bx;F,*}) := (Y^{t,\bx;\tau^{F,*}(X^{t,\bx})},Z^{t,\bx;\tau^{F,*}(X^{t,\bx})},V^{t,\bx;\tau^{F,*}(X^{t,\bx})},K^{-,t,\bx;\tau^{F,*}(X^{t,\bx})})
\end{align*}
is the unique maximal solution in $\mcS^2_{[t,\tau^{F,\zeta}_t(X^{t,\bx})]}\times \mcH^2_{[t,\tau^{F,\zeta}_t(X^{t,\bx})]}(W)\times \mcH^2_{[t,\tau^{F,\zeta}_t(X^{t,\bx})]}(\mu)\times\mcK^2_{[t,\tau^{F,\zeta}_t(X^{t,\bx})]}(\mu)$ to the BSDE
\begin{align*}
  \begin{cases}
    Y^{t,\bx;F,*}_s=v(\tau^{F,\zeta}(X^{t,\bx}),X^{t,\bx})+\int_s^{\tau^{F,\zeta}(X^{t,\bx})} f(r,X^{t,\bx},I^t_r)dr-\int_s^{\tau^{F,\zeta}(X^{t,\bx})} Z^{t,\bx;F,*}_r dW_r
    \\
    \quad-\int_s^{\tau^{F,\zeta}(X^{t,\bx})}\!\!\int_E V^{t,\bx;F,*}_r(e)\mu(dr,de) -( K^{-,t,\bx;F,*}_{\tau^{F,\zeta}(X^{t,\bx})}-K^{-,t,\bx;F,*}_s),\quad\forall s\in [t,{\tau^{F,\zeta}(X^{t,\bx})}]
    \\
    V^{t,\bx;F,*}_s(e)\leq 0,\quad d\Prob\otimes ds\otimes \lambda(de)-\text{a.e.}
  \end{cases}
\end{align*}
In light of Theorem 3.1 in \cite{Kharroubi2010}, this means that on $[t,{\tau^{F,\zeta}(X^{t,\bx})}]$ the process $Y^{t,\bx;F,*}$ is the pointwise limit of the non-increasing sequence $(Y^{\zeta,n})_{n\in\bbN}$, where for each $n\in\bbN$, the process $Y^{\zeta,n}$ is the first component in the triple $(Y^{\zeta,n},Z^{\zeta,n},V^{\zeta,n})\in \mcS^2_{[t,\tau^{F,\zeta}(X^{t,\bx})]}\times \mcH^2_{[t,\tau^{F,\zeta}(X^{t,\bx})]}(W)\times \mcH^2_{[t,\tau^{F,\zeta}(X^{t,\bx})]}(\mu)$ defined as the unique solution to the BSDE
\begin{align*}
    Y^{\zeta,n}_s&=v(\tau^{F,\zeta}_t(X^{t,\bx}),X^{t,\bx})+\int_s^{\tau^{F,\zeta}_t(X^{t,\bx})} f(r,X^{t,\bx},I^t_r)dr-\int_s^{\tau^{F,\zeta}(X^{t,\bx})} Z^{\zeta,n}_r dW_r
    \\
    &\quad-\int_s^{\tau^{F,\zeta}_t(X^{t,\bx})}\!\!\int_E V^{\zeta,n}_r(e)\mu(dr,de) - n\int_s^{\tau^{F,\zeta}_t(X^{t,\bx})}\!\!\int_U(V^{\zeta,n}_r(e))^-\lambda(de)dr,\quad\forall s\in [t,{\tau^{F,\zeta}(X^{t,\bx})}].
\end{align*}
On the other hand, using the comparison principle as in the proof of Lemma 3.4 in \cite{imp-stop-game} yields that
\begin{align*}
  Y^{\zeta,n}_t=\essinf_{\nu\in\mcV^n}\E^\nu\Big[v(\tau^{F,\zeta}(X^{t,\bx}),X^{t,\bx}) + \int_t^{\tau^{F,\zeta}(X^{t,\bx})}f(s,X^{t,\bx},I^t_s)ds\,\Big|\,\mcF_t\Big].
\end{align*}
Taking the limit as $n\to\infty$, we arrive at the dynamic programming relation
\begin{align*}
  v^{\mcR,F,*}(t,\bx)=\essinf_{\nu\in\mcV}\E^\nu\Big[v(\tau^{F,\zeta}(X^{t,\bx}),X^{t,\bx}) + \int_t^{\tau^{F,\zeta}(X^{t,\bx})}f(s,X^{t,\bx},I^t_s)ds\,\Big|\,\mcF_t\Big].
\end{align*}
Consequently,
\begin{align*}
  |v^{\mcR,F,*}(t,\bx) - v^{\mcR,F,\zeta}(t,\bx)|\leq\esssup_{\nu\in\mcV}\E^\nu\Big[|v(\tau^{F,\zeta}(X^{t,\bx}),X^{t,\bx})- \psi(\tau^{F,\zeta}(X^{t,\bx}),X^{t,\bx})|\,\Big|\,\mcF_t\Big]\leq \zeta
\end{align*}
and letting $\zeta\to 0$, the result follows.
\end{proof}

\subsection{An auxiliary probability space\label{subsec:proof-of-prop}}

Inspired by Section 4.3 of \cite{Fuhrman2020}, we introduce an auxiliary probability space $( \Omega',\mcF',\Prob')$ on which lives real-valued random variables $(U^m_j,S^m_j)_{m,j\in\bbN}$ and random measures $(\pi^l)_{l\in\bbN}$ such that
\begin{itemize}
  \item the $U^m_j$ are all uniformly distributed on $(0,1)$,
  \item the probability distribution of $S^m_j$ admits a density $f^m_j$ with respect to the Lebesgue measure, that has support on the interval $((1-2^{1-j})/m,(1-2^{-j})/m)$, so that $0<S^m_1<S^m_2<\cdots<1/m$ for every $m\in\bbN$,
  \item every $\pi^l$ is a Poisson random measure on $(0,\infty)\times U$, with compensator $l^{-1}\lambda(da)dt$, with respect to its natural filtration;
  \item the random elements $U^m_j,S^{m'}_{j'},\pi^l$ are all independent.
\end{itemize}
Now, we define $\hat\Omega:=\Omega\times\Omega'$, let $\hat \mcF$ be the $\Prob\otimes\Prob'$ completion of $\mcF\otimes\mcF'$ and let $\hat\Prob$ denote the extension of $\Prob\otimes\Prob'$ to $\hat\mcF$. Further, we let $\hat W,\hat\mu,\hat U^m_j,\hat S^{m'}_{j'}$ and $\hat \pi^l$ denote the canonical extensions of $W,\mu,U^m_j,S^{m'}_{j'}$ and $\pi^l$ to $\hat\Omega$. For $u\in\check\mcU_t$ and $\tau\in\check\mcT_t$ (which are extensions of $\mcU_t$ and $\mcT_t$ to $\hat\Omega$, that are more carefully defined below), we let
\begin{align*}
  \hat J^{F,\zeta}(t,\bx;u)=\hat\E\Big[\psi(\tau^{F,\zeta}_t(\hat X^{t,\bx;u}),\hat X^{t,\bx;u})+\int_t^{\tau} f(r,\hat X^{t,\bx;u},u_r)dr\,\Big|\,\hat\mcF_t\Big],
\end{align*}
where $\hat\E$ is expectation with respect to $\hat\Prob$, the filtration $\hat\bbF:=(\hat \mcF_t)_{{t\geq 0}}$ is the $\hat\Prob$-augmented natural filtration on $(\hat\Omega,\hat\mcF)$ generated by $\hat W$ and $\hat X^{t,\bx;u}$ solves
\begin{align*}
  \hat X^{t,\bx;u}_s=\bx(s\wedge t)+\int_t^{t\vee s} a(r,\hat X^{t,\bx;u},u_r)dr+\int_t^{t\vee s}\sigma(r,\hat X^{t,\bx;u},u_r)d\hat W_r,\quad\forall s\in [0,T].
\end{align*}
When extending the basic notations to the probability space $(\hat\Omega,\hat\mcF,\hat\Prob)$, we introduce two versions of most objects depending on whether they utilize the information in the $\sigma$-algebra $\mcF'$ or not. We denote objects that incorporate the information in $\mcF'$ with a check symbol, while objects that do not use this information are denoted with a hat symbol. Specifically, we make the following definitions:
\begin{itemize}
\item We let $\hat\bbF$ (resp.~$\check\bbF$) be the $\hat\Prob$-completion of the filtration $(\mcF_s \times \Omega')_{s\geq 0}$ (resp.~$(\mcF_s\otimes \mcF')_{s\geq 0}$).
\item We let $\hat\mcT_t$ (resp.~$\check\mcT_t$) be the set of all $\hat\bbF$-stopping times (resp.~$\check\bbF$-stopping times) $\tau$ with $\tau\in [t,T]$, $\hat\Prob$-a.s.
\item We let $\hat\mcU_t$ (resp. $\check\mcU_t$) be the set of all $\Prog(\hat\bbF)$-measurable (resp.~$\Prog(\check\bbF)$-measurable) processes $(u_s:t\leq s\leq T)$ valued in $U$.
\end{itemize}

Following the above procedure, we define $\hat v^{F,*}(t,\bx)$ as the canonical extension of $v^{F,*}(t,\bx)$ to $\hat\Omega$. We fix $(t,\bx)\in [0,T]\times\bfC^d$ and note that for any $\varrho>0$, there is a $\hat u^{\varrho}\in\hat\mcU_t$ such that
\begin{align}\label{ekv:u-varrho-eps-def}
  \hat{v}^{F,*}(t,\bx)\geq \hat J^{F,*}(t,\bx;\hat u^{\varrho})-\varrho.
\end{align}

The idea is to first approximate $\hat u^{\varrho}$ by a discretized version, $\hat u^{\varrho,\eps}$, and then use the sequences $\hat U^m_j$ and $\hat S^{m'}_{j'}$ to ``randomize'' $\hat u^{\varrho,\eps}$ and add the jumps in $\hat \pi^l$ to obtain a point process $\check u\in\check\mcU$ such that the $\hat\Prob$-compensator of the corresponding random measure, \ie the unique random measure $\check\mu$ (that separates points in time) such that $\check u_s=\beta_0+\int_0^s\int_U(e-\check u_r)\check\mu(dr,de)$), has a density $\check\nu$ with respect to $\lambda(da)dt$ which is bounded from below by a positive constant and such that $\check u$ is sufficiently close to $\hat u$ under the $\hat\rho_t$.

Consequently, as in \cite{cont-stop-P1}, we introduce the following discretization:
\begin{defn}\label{def:discretization}
For each $\eps>0$:
\begin{itemize}
  \item We let $n^\eps\geq 0$ be the smallest integer such that $2^{-n^\eps}(T-t)\leq\eps$, set $n_{\bbT}^{t,\eps}:=2^{n^\eps}+1$ and introduce the discrete set $\bbT^{t,\eps}:=\{t^\eps_i:t^\eps_i=t+(i-1)2^{-n^\eps}(T-t),i=1,\ldots,n_{\bbT}^\eps\}$, a discretization of $[t,T]$ with step-size $\Delta^{t,\eps}:=2^{-n^\eps}(T-t)$. For $s\in [t,T]$, we let $\bbT^{t,\eps}_s:=\bbT^{t,\eps}\cap[s,T]$.
  \item We let $(U^\eps_{i})_{i=1}^{n^\eps_U}$ be a Borel-partition of $U$ such that each $U^\eps_i$ has non-empty interior in $U$ and a diameter that does not exceed $\eps$ and let $(b^\eps_i)_{i=1}^{n^\eps_U}$ be a sequence with $b^\eps_i\in \text{int}\,U^\eps_i$ and denote by $\bar U^\eps:=\{b^\eps_1,\ldots,b^\eps_{n^\eps_U}\}$ the corresponding discretization of $U$.
\end{itemize}
\end{defn}
Moreover, we define a corresponding discretization of the control set by letting
\begin{align*}
\mcU^{\eps}_t:=\Big\{\Big( \sum_{i=1}^{n_{\bbT}^\eps-2}\beta_{i}\ett_{[t^\eps_i,t^\eps_{i+1})}(s)+\beta_{n_{\bbT}^\eps-1}\ett_{[t^\eps_{n_{\bbT}^\eps-1},T]}(s): t\leq s\leq T\Big)\, : \,(\beta_i:\Omega\mapsto \bar U^\eps)\in m\mcF_{t^\eps_i}\Big\}.
\end{align*}
We introduce the projection $\Xi_\mcU^\eps:\mcU_t\to\mcU^\eps_t$ defined for each $u\in\mcU_t$ as
\begin{align*}
  \Xi_\mcU^\eps[u](s):= \sum_{i=1}^{n_{\bbT}^\eps-2}b^\eps_{\iota^\eps_i(u)}\ett_{[t^\eps_i,t^\eps_{i+1})}(s) + b^\eps_{\iota^\eps_{n_{\bbT}^\eps-1}(u)}\ett_{[t^\eps_{n_{\bbT}^\eps-1},T]}(s),
\end{align*}
where $\iota^\eps_i(u)$ is a measurable selection of
\begin{align*}
  \iota^\eps_i(u)\in\argmin_{j\in\{1,\ldots,n^\eps_U\}} \text{dist}\Big(U^\eps_j,\frac{1}{\Delta^{t,\eps}}\E\Big[\int_{t^\eps_i}^{t^\eps_{i+1}}u_sds\,\Big|\,\mcF_{t^\eps_i}\Big]\Big),
\end{align*}
with
\begin{align*}
  \text{dist}(A,x):=\inf\{|y-x|:y\in A\}.
\end{align*}
We let $\hat u^{\varrho,\eps}:=\Xi_\mcU^\eps(\hat u^{\varrho})$ and define an equivalent of the pseudo-metric $\rho$ on $\check\mcU_t$, by introducing
\begin{align*}
  \hat\rho_t(u,\tilde u):=\hat\E\Big[\int_{t}^T|u_s-\tilde u_s| ds\Big].
\end{align*}
The following result is a direct consequence of density of the set of piecewise constant adapted processes in the set of progressively measurable processes under the $L^1$-norm and the compactness of $U$.
\begin{lem}\label{lem:dens}
  For any $\varrho>0$, we have $\rho_t(\hat u^{\varrho},\hat u^{\varrho,\eps})\to 0$, as $\eps\to 0$.
\end{lem}

Utilizing the discretized control, $\hat u^{\varrho,\eps}$, we define a family of random measures in the probability space $(\hat\Omega,\hat\mcF,\hat\Prob)$ as follows:

\begin{lem}\label{lem:will-conv}
For any $\varrho>0$, there is family $(\check u^{\varrho,\eps})_{\eps>0}$ of piecewise constant process
\begin{align*}
  \check u^{\varrho,\eps}:=\beta_0\ett_{[t,\eta^{\varrho,\eps}_1)}+\sum_{j\geq 1}\theta_j\ett_{[\eta^{\varrho,\eps}_j,\eta^{\varrho,\eps}_{j+1})}\in\check\mcU_t,
\end{align*}
with $(\eta^{\varrho,\eps}_j)_{j\in\bbN}$ strictly increasing, such that
\begin{align*}
  \hat\rho_t(\check u^{\varrho,\eps},\hat u^{\varrho})\to 0\quad \text{as }\eps\to 0.
\end{align*}
Moreover, for each $\eps>0$ the random measure on $[t,T]\times U$ corresponding to $\check u^{\varrho,\eps}$, \ie $\check\mu^{{\varrho,\eps}}:=\sum_{j\geq 1} \delta_{(\eta_j,\theta_j)}$ has a $\hat\Prob$-compensator with respect to the filtration $\check\bbF^{\hat W,\check u^{\varrho,\eps}}$ that is absolutely continuous with respect to $\lambda$ and takes the form
\begin{align*}
\check\nu^{\varrho,\eps}_s(\hat\omega,e)\lambda(de)ds
\end{align*}
where $\check\nu^{\varrho,\eps}$ is $\Pred(\check\bbF^{\hat W,\check u^{\varrho,\eps}}) \otimes\mcB(U)$-measurable and bounded away from zero.
\end{lem}

\begin{proof}
For each $m\geq 1$, define the kernel $q^m:(b,da)\mapsto \frac{1}{\lambda({\mathbf B}(b,1/m))}\ett_{{\mathbf B}(b,1/m)}(a)\lambda(da)$ (where ${\mathbf B}(b,1/m)$ is the closed ball of radius $1/m$, centered at $b$) as in the proof of Lemma 4.4 of \cite{Fuhrman2020}. Using the sequence $(q^m)_{m\in\bbN}$, we define the sequence of controls $(\check u^{\varrho,\eps,m})_{m\in\bbN}$ as the piecewise constant processes
\begin{align*}
  \check u^{\varrho,\eps,m}:=\beta_0\ett_{[t,\check\eta^{\varrho,\eps,m}_1)}+\sum_{j\geq 1}\check\theta^{\varrho,\eps,m}_j\ett_{[\check\eta^{\varrho,\eps,m}_j,\check\eta^{\varrho,\eps,m,l}_{j+1})},
\end{align*}
where
\begin{align*}
 \begin{cases}
  \check \eta^{\varrho,\eps,m}_1:=t+\hat S^m_1
  \\
  \check \eta^{\varrho,\eps,m}_j:=(\check \eta^{\varrho,\eps,m}_{j-1}\vee t^\eps_j)+\hat S^m_j,\quad j>1,
  \\
  \check\theta^{\varrho,\eps,m}_j:=q^m(\hat u^{\varrho,\eps}_{t^\eps_j},\hat U^m_j),\quad\forall j\in\bbN.
 \end{cases}
\end{align*}

An important feature of the above definition is that $\check u^{\varrho,\eps,m}_s$ lies within a $1/m$-neighborhood of $\hat u^{\varrho,\eps}_s$ whenever $s\in \cup_{j=1}^{n^{t,\eps}_\bbT-1}[\check\eta^{\varrho,\eps,m}_j,t^\eps_{j+1})$. Since $\check\eta^{\varrho,\eps,m}_j\searrow t^\eps_j$, $\hat\Prob$-a.s.,~as $m\to\infty$, we conclude that
\begin{align*}
  \check\rho_t(\check u^{\varrho,\eps,m},\hat u^{\varrho,\eps})\to 0,\quad\text{as }m\to\infty.
\end{align*}

The above control induces a random measure on $[t,T]\times U$ defined as,
\begin{align*}
  \mu^{{\varrho,\eps,m}}:=\sum_{j=1}^{n^{t,\eps}_\bbT-1}\delta_{(\check \eta^{\varrho,\eps,m}_j,\check \theta^{\varrho,\eps,m}_j)}.
\end{align*}
According to Lemma A.11 in~\cite{Fuhrman15}, the random measure $\mu^{{\varrho,\eps,m}}$ has a $\hat\Prob$-compensator with respect to $\check\bbF^{\hat W,\check u^{\varrho,\eps,m}}$ given by the explicit formula
\begin{align*}
  \sum_{j=1}^{n^{t,\eps}_\bbT-1}\ett_{(\check \eta^{\varrho,\eps,m}_{j-1}\vee t^\eps_j ,\check \eta^{\varrho,\eps,m}_j]}(t)q^m(\hat u^{\varrho,\eps}_{t^\eps_j},da)\frac{f^m_j(s-(\check \eta^{\varrho,\eps,m}_{j-1}\vee t^\eps_j))}{1-F^m_j(s-(\check \eta^{\varrho,\eps,m}_{j-1}\vee t^\eps_j))}ds,
\end{align*}
with $F^m_j(s):=\int_{-\infty}^s f^m_j(r)dr$. For each $m\in\bbN$, this compensator is clearly $\Pred(\check\bbF^{\hat W,\check u^{\varrho,\eps,m}}) \otimes\mcB(U)$-measurable. However, the density equals zero on $\cup_{j=1}^{n^{t,\eps}_\bbT-1}[\check\eta^{\varrho,\eps,m}_j,t^\eps_{j+1})$ and is, therefore, not bounded away from zero. To remedy this we add the jumps in $\hat\pi^l$ to obtain the random measure $\check\mu^{\varrho,\eps,m,l}:=\mu^{\varrho,\eps,m}+\hat\pi^l$ that corresponds to the randomized control
\begin{align*}
  \check I^{m,l}_s=\beta_0+\int_t^s\!\!\int_U(e-\check I^{m,l}_{r-})\check\mu^{\varrho,\eps,m,l}(dr,de),\quad\forall s\in [t,T].\\
\end{align*}
With this definition, $\check\mu^{\varrho,\eps,m,l}$ has $\hat\Prob$-compensator with respect to the filtration $\check\bbF^{\hat W,\check I^{m,l}}$ that is absolutely continuous with respect to $\lambda$ and takes the form
\begin{align*}
\check\nu^{m,l}_s(\hat\omega,e)\lambda(de)ds
\end{align*}
where $\check\nu^{m,l}$ is $\Pred(\check\bbF^{\hat W,\check I^{m,l}}) \otimes\mcB(U)$-measurable and bounded from below by $1/l$. On the other hand, combining Lemma~\ref{lem:dens} with the above construction yields that
\begin{align*}
  \check\rho_t(\check I^{m,l},\hat u^{\varrho,\eps})\to 0,\quad\text{as }m,l\to\infty.
\end{align*}
This proves that for certain maps $m,l:(0,\infty)\to \bbN$, the family $(\check\mu^{\varrho,\eps,m(\eps),l(\eps)})_{\eps >0}$ fulfills the assertion.
\end{proof}

To limit notation we drop the superscripts and set $\check u:=\check u^{\varrho,\eps}$ and $\check\nu=\check\nu^{\varrho,\eps}$. To establish a correspondence between the primal control problem terminated at $\tau^{F,\zeta}(X^{t,\bx;u})$ and its dual (randomized) version, we combine $\pi_1$ and $\mu^{\varrho,\eps}$ to obtain the random measure $\check\mu:=\pi_1(\cdot\cup [0,t],\cdot)+\mu^{\varrho,\eps}(\cdot\cup (t,T],\cdot)$. By Lemma~\ref{lem:will-conv}, $\check\mu$ has a $\hat\Prob$-compensator with respect to the filtration $\hat\bbF^{\mcR}$, the latter being the natural filtration on $(\hat\Omega,\check\mcF)$ generated by $\hat W$ and $\check u$, completed with all $\hat\Prob$-null sets. Moreover, this compensator has a density $\ett_{[0,t]}+\ett_{(t,T]}\check\nu$ with respect to $\lambda$. We abuse notation and use $\check \nu$ to denote this density, the infimum of which is strictly positive whereas the supremum may be unbounded.

The above construction allows us to define an auxiliary randomized version of the game. We denote by $\check\bbF^{\mcR}:=\check\bbF^{\hat W,\check u}$ the filtration generated by $\hat W$ and $\check u$, augmented with all $\hat\Prob$-null sets. Letting $(\check\sigma_j,\check\zeta_j)_{j\geq 1}$ be the marks of $\check\mu$ we find, since $\check\nu$ is bounded from below, that
\begin{align*}
\hat M_s:=\exp\Big(\int_{0}^s\!\int_U(1-(\check\nu_r(a))^{-1})\lambda(da)dr\Big)\prod_{\check\sigma_j\leq s}(\check\nu_{\check\sigma_j}(\check\zeta_j))^{-1}
\end{align*}
is a strictly positive martingale with respect to the filtration $\check\bbF^{\mcR}$ under $\hat\Prob$. Furthermore, as $\check\nu_s\equiv 1$ for all $s\in [0,t]$, we have $\hat M\equiv 1$ on $[0,t]$. We define the equivalent probability measure $\check\Prob$ on $(\hat\Omega,\hat\mcF)$ as $d\check\Prob=\hat M_Td\hat\Prob$. By the Girsanov theorem, $\check\mu$ has $\check\Prob$-compensator $\lambda(da)ds$ with respect to the filtration $\check\bbF^\mcR$. Moreover, despite the fact that $\check\nu$ is generally not bounded we still have a Dol{\'e}ans-Dade exponential
\begin{align*}
\hat\kappa^{\check\nu}_s:=\exp\Big(\int_{t}^s\!\int_U(1-\check\nu_r(e))\lambda(de)dr\Big)\prod_{t<\check\sigma_j\leq s}\check\nu_{\check\sigma_j}(\check\zeta_j)
\end{align*}
for which $\check \E[\hat\kappa^{\check\nu}_T]=\hat \E[\hat M_T\hat\kappa^{\check\nu}_T]=1$, proving that $\hat\kappa^{\check\nu}$ is a $\check\Prob$-martingale. We can thus define a corresponding probability measure, $\check\Prob^{\check\nu}$, on $(\hat\Omega,\hat\mcF)$ as $d\check\Prob^{\check\nu}:=\hat\kappa^{\check\nu}_Td\check\Prob$, and since $\hat M_T\hat\kappa^{\check\nu}_T\equiv 1$, we conclude that $\check\Prob^{\check\nu}=\hat\Prob$ on $(\hat\Omega,\hat\mcF)$. We further extend this definition by letting $d\check\Prob^{\nu}:=\hat\kappa^{\nu}_Td\check\Prob$ whenever $\nu\in\check\mcV$. Here, $\check\mcV$ is the set of all $\check\bbF^\mcR$-predictably measurable bounded maps $\nu=\nu_t(\hat\omega,e):[0,T]\times\hat\Omega\times U\to [0,\infty)$. In particular, with
\begin{align*}
  \begin{cases}
    \check I^t_s=\beta_0+\int_t^s\!\!\int_U(e-\check I^t_{r-})\check \mu(dr,de),\quad\forall s\in [t,T],\\
    \check X^{t,\bx}_s=\bx(s\wedge t)+\int_t^{s\vee t} a(r,\check X^{t,\bx},\check I^t_r)dr+\int_t^{s\vee t}\sigma(r,\check X^{t,\bx},\check I^t_r)d\hat W_r,\quad\forall s\in [0,T]
  \end{cases}
\end{align*}
we get that $\hat J^{F,\zeta}(t,\bx;\check u)=\check J^{\mcR,F,\zeta}(t,\bx;\check\nu)$, $\check\Prob$-a.s., where
\begin{align*}
\check J^{\mcR,F,\zeta}(t,\bx;\check\nu) := \check\E^{\nu}\Big[\psi(\tau^{F,\zeta}(\check X^{t,\bx}),\check X^{t,\bx})+\int_t^{\tau^{F,\zeta}(\check X^{t,\bx})} f(r,\check X^{t,\bx},\check I^t_r)dr\,\Big|\,\hat\mcF_t\Big]
\end{align*}
and $\check\E^{\nu}$ is expectation with respect to $\check\Prob^\nu$.

Since the probability space $(\hat\Omega,\hat\mcF,\check\Prob,\hat W,\check\mu)$ is a setting for our penalized BSDEs \eqref{ekv:rbsde-pen}, there is a unique quadruple $(\check Y^{t,\bx,\zeta,n},\check Z^{t,\bx,\zeta,n},\check V^{t,\bx,\zeta,n},\check K^{-,{t,\bx,\zeta,n}})\in \check\mcS^2\times\check\mcH^2(\hat W)\times\check\mcH^2(\check\mu)\times \check\mcA^2$, where $\check\mcS^2$, $\check\mcH^2(\hat W)$, $\check\mcH^2(\check \mu)$ and $\check\mcA^2$ are defined as $\mcS^2$, $\mcH^2(W)$, $\mcH^2(\mu)$ and $\mcA^2$ but on the probability space $(\hat\Omega,\hat\mcF,\check\Prob,\hat W,\check\mu)$, such that
\begin{align}\nonumber
    \check Y^{t,\bx,\zeta,n}_s&=\psi(\tau^{F,\zeta}(\check X^{t,\bx}),\check X^{t,\bx})+\int_s^{\tau^{F,\zeta}(\check X^{t,\bx})} f(r,\check X^{t,\bx},\check I^t_r)dr-\int_s^{\tau^{F,\zeta}(\check X^{t,\bx})} \check Z^{t,\bx,\zeta,n}_r d\hat W_r
    \\
    &\quad -\int_s^{\tau^{F,\zeta}(\check X^{t,\bx})}\!\!\!\int_U \check V^{t,\bx,\zeta,n}_r(e)\check\mu(dr,de)-n\int_s^T\!\!\!\int_U(\check V^{t,\bx,\zeta,n}_r(e))^-\lambda(de)dr,\quad \forall s\in [t,{\tau^{F,\zeta}(\check X^{t,\bx})}].\label{ekv:rbsde-pen-hat}
\end{align}
By standard results for BSDEs with jumps, we find that $v^{\mcR,F,\zeta,n}\searrow v^{\mcR,F,\zeta}$, where
\begin{align}
v^{\mcR,F,\zeta,n}(t,\bx)&=\essinf_{\nu\in\check\mcV^n}\check J^{\mcR,F,\zeta}(t,\bx;\nu)=\check Y^{t,\bx,n}_t,\label{ekv:Yn-repr-hat}
\end{align}
employing the obvious notation $\check\mcV^n:=\{\nu\in\check\mcV:\nu\leq n\}$.

\subsection{Relating the primal and dual control problems}
The characteristics of $\check u^{\varrho,\eps}$ given in the statement of Lemma~\ref{lem:will-conv} allow us to prove the following lemma which is central in the proof of Proposition~\ref{prop:dual-CP}:
\begin{lem}\label{lem:tau-zeta-squeeze}
For any $\zeta>0$, we have
\begin{align*}
    \liminf_{\eps\to 0}\hat\Prob\big[\{\omega\in\Omega:\tau^{F,2\zeta}(\hat X^{t,\bx;\hat u^{\varrho}})\leq \tau^{F,\zeta}(\hat X^{t,\bx;\check u^{\varrho,\eps}})\leq \tau^{F,*}_t(\hat X^{t,\bx;\hat u^{\varrho}})\}\big]=1.
\end{align*}
\end{lem}

\begin{proof}
By construction we have $\check\rho_t(\hat u^{\varrho},u^{\varrho,\eps})$. Since \eqref{ekv:Xu-stab} readily extends to the describe stability of the state when control sequences converge under $\check\rho$, the control $\check u^{\varrho,\eps}$ approximates $\hat u^{\varrho}$ in the sense that the corresponding state processes satisfy
\begin{align*}
  \sup_{s\in[t,T]}|\hat X^{t,\bx;\check u^{\varrho,\eps}}_s-\hat X^{t,\bx;\hat u^{\varrho}}_s|\;\to\;0
\quad\text{in }\hat\Prob\text{-probability as }\eps\to 0.
\end{align*}
Since $v$ is continuous and the interval $[0,T]$ is compact, it follows that $\by\mapsto v(\cdot,\by):\bfC^d\to\bfC^1$ is continuous and, hence,
\begin{align*}
\sup_{s\in[t,T]}|v(s,\hat X^{t,\bx;\check u^{\varrho,\eps}})-v(s,\hat X^{t,\bx;\hat u^{\varrho}})|\;\to\;0\quad\text{in }\hat\Prob\text{-probability as }\eps\to 0.
\end{align*}
Fix $\zeta>0$. For and $\delta>0$ we thus have that for any sufficiently small $\eps>0$,
\begin{align*}
\hat\Prob\Big[\sup_{s\in[t,T]}|v(s,\hat X^{t,\bx;\check u^{\varrho,\eps}})-v(s,\hat X^{t,\bx;\hat u^{\varrho}})|>\zeta/2\Big]\le \delta.
\end{align*}
On the other hand, for any $K > 0$ on the set $\{(\bx,\tilde\bx)\in\bfC^d\times\bfC^d:\|\bx\|_T\vee\|\tilde\bx\|_T\leq K\}$ we have
\begin{align*}
  \sup_{s\in[0,T]}|\psi(s,\bx)-\psi(s,\tilde\bx)|\leq \varpi_K(\|\bx-\tilde\bx\|_T).
\end{align*}
Hence,
\begin{align*}
\sup_{s\in[t,T]}\ett_{[\|\hat X^{t,\bx;\check u^{\varrho,\eps}}\|_T\vee\|\hat X^{t,\bx;\hat u^{\varrho}}\|_T\leq K]}|\psi(s,\hat X^{t,\bx;\check u^{\varrho,\eps}}) - \psi(s,\hat X^{t,\bx;\hat u^{\varrho}})|\;\to\;0
\quad\text{in probability}.
\end{align*}
Using uniform integrability, we find that for any $\eps>0$ that is sufficiently small we have
\begin{align*}
\check\Prob\Big[\sup_{s\in[t,T]}|\psi(s,\hat X^{t,\bx;\check u^{\varrho,\eps}})-\psi(s,\hat X^{t,\bx;\hat u^{\varrho}})|>\zeta/2\Big]\leq \delta.
\end{align*}
In particular, with $g:=v-\psi$ we find that
\begin{align*}
\check\Prob\Big[\sup_{s\in[t,T]}|g(s,\hat X^{t,\bx;\check u^{\varrho,\eps}})-g(s,\hat X^{t,\bx;\hat u^{\varrho}})|>\zeta\Big]\leq \delta.
\end{align*}

On the complement of this event, the trajectories of $g(\cdot,\hat X^{t,\bx;\check u^{\varrho,\eps}})$ and $g(\cdot,\hat X^{t,\bx;\hat u^{\varrho}})$ remain within a distance $\zeta$ from each other, and by the definition of the hitting times, this implies that
\begin{align*}
\tau^{F,2\zeta}(\hat X^{t,\bx;\hat u^{\varrho}}) \leq \tau^{F,\zeta}(\hat X^{t,\bx;\check u^{\varrho,\eps}}) \leq \tau^{F,*}_t(\hat X^{t,\bx;\hat u^{\varrho}}).
\end{align*}
Since $\delta>0$ was arbitrary, the desired result follows.
\end{proof}

\medskip

We can now use the fact that $\tau^{2\zeta}(X^{t,\bx;\hat u^{\varrho}}) \nearrow \tau^{F,*}_t(X^{t,\bx;\hat u^{\varrho}})$ as $\zeta\searrow 0$ to relate the expected value related to the $\varrho$-optimal control $\hat u^\varrho$ for the control problem with horizon $\tau^{F,*}$ to the value corresponding to the randomized control $\check u^{\varrho,\eps}$ in the problem with horizon $\tau^{F,\zeta}$.

\begin{lem}\label{lem:hat-zeta-approx}
 There is a $\zeta_0>0$ such that for any $\zeta\in (0,\zeta_0)$, there is a corresponding $\eps>0$ such that
 \begin{align*}
   \check\E\big[\hat J^{F,\zeta}(t,\bx;\check u^{\varrho,\eps}) - \hat J^{F,*}(t,\bx;\hat u^\varrho)\big] \leq \varrho.
 \end{align*}
\end{lem}

\begin{proof}
For any $\zeta,\eps>0$, we have
\begin{align*}
  &\hat J^{F,\zeta}(t,\bx;\check u^{\varrho,\eps})-\hat J^{F,*}(t,\bx;\hat u^\varrho)
  \\
  &\leq \check\E\Big[\psi(\tau^{F,\zeta}(\hat X^{t,\bx;\check u^{\varrho,\eps}}),\hat X^{t,\bx;\check u^{\varrho,\eps}})-\psi(\tau^{F,*}_t(\hat X^{t,\bx;\hat u^\varrho}),\hat X^{t,\bx;\hat u^\varrho})
  \\
  &\quad+\int_t^{\tau^{F,*}_t(\hat X^{t,\bx;\hat u^\varrho})\wedge \tau^{F,\zeta}(\hat X^{t,\bx;\check u^{\varrho,\eps}})} \big(f(s,\hat X^{t,\bx;\check u^{\varrho,\eps}},\check u^{\varrho,\eps}_s)-f(s,\hat X^{t,\bx;\hat u^\varrho},\hat u^\varrho_s)\big)ds
  \\
  &\quad+\int_{\tau^{F,*}_t(\hat X^{t,\bx;\hat u^\varrho})\wedge \tau^{F,\zeta}(\hat X^{t,\bx;\check u^{\varrho,\eps}})}^{\tau^{F,*}_t(\hat X^{t,\bx;\hat u^\varrho})\vee \tau^{F,\zeta}(\hat X^{t,\bx;\check u^{\varrho,\eps}})} (|f(s,\hat X^{t,\bx;\hat u^\varrho},\hat u^\varrho_s)|+|f(s,\hat X^{t,\bx;\check u^{\varrho,\eps}},\check u^{\varrho,\eps}_s)|)ds\,\Big|\,\mcF_t\Big].
\end{align*}
For every $\zeta$, we can now choose $\eps(\zeta)$ such that Lemma~\ref{lem:tau-zeta-squeeze} holds for $(\zeta,\eps(\zeta))$ and $\eps(\zeta)\to 0$ as $\zeta\to 0$. Then, left upper semi-continuity of $\psi$ and continuity of $\psi$ and $f$ in $\bx$ guarantee that
\begin{align*}
  \limsup_{\zeta\to 0} \check\E\big[\hat J^{F,\zeta}(t,\bx;\check u^{\varrho,\eps(\zeta)})-\hat J^{F,*}(t,\bx;\hat u^\varrho)\big] \leq 0,
\end{align*}
as $\zeta\to 0$.
\end{proof}

\medskip

Clearly, $\hat J^{F,\zeta}(t,\bx;\check u^{\varrho,\eps})=\check J^{\mcR,F,\zeta}(t,\bx;\check\nu^{\varrho,\eps})$. However, $\check\nu^{\varrho,\eps}$ is not necessarily bounded and, therefore, the above lemma does not immediately relate the value of the primal control problem to the value of its dual counterpart. On the other hand, the following lemma proves that in terms of the corresponding value function, $\check\nu^{\varrho,\eps}$ can be well approximated by a bounded density.

\begin{lem}\label{lem:check-nu-approx}
For any $\zeta\geq 0$, we have that $\check J^{\mcR,F,\zeta}(t,\bx;\check\nu^n)\to\check J^{\mcR,F,\zeta}(t,\bx;\check\nu)$, $\Prob$-a.s.~as $n\to\infty$.
\end{lem}

\begin{proof}
Letting
\begin{align*}
  \Phi_t(\tau):=\psi(\tau,\check X^{t,\bx})+\int_t^{\tau} f(r,\check X^{t,\bx},\check I^{t}_r)dr,\quad\forall \tau\in\check\mcT_t
\end{align*}
we have
\begin{align*}
  |\Phi_t(\tau)|\leq C(1+\|\check X^{t,\bx}\|^q_T)=:\bar\Phi.
\end{align*}
and get that
\begin{align*}
  |\check J^{\mcR,F,\zeta}(t,\bx;\check\nu^n)-\check J^{\mcR,F,\zeta}(t,\bx;\check\nu)|&\leq \check\E\big[|(\hat\kappa^{\check\nu^n}_T-\hat\kappa^{\check\nu}_T)\Phi_t(\tau^{F,\zeta}(\check X^{t,\bx}))|\,\big|\,\hat\mcF_t\big]
  \\
  &\leq \check\E\big[|\hat\kappa^{\check\nu^n}_T-\hat\kappa^{\check\nu}_T|\bar\Phi\,\big|\,\hat\mcF_t\big].
\end{align*}
We will show that the right-hand side tends to 0, $\check\Prob$-a.s., as $n\to\infty$. Letting $E_K:=\{\omega:\|\check X^{t,\bx}\|_T\leq K\}$, we get
\begin{align*}
  \check\E\big[|\hat\kappa^{\check\nu^n}_T-\hat\kappa^{\check\nu}_T|\bar\Phi\big|\hat\mcF_{t}\big]\leq \check\E\big[\ett_{E_K}|\hat\kappa^{\check\nu^n}_T-\hat\kappa^{\check\nu}_T|\bar\Phi\big|\hat\mcF_{t}\big] + \check\E^{\check\nu^n}\big[\ett_{E_K^c}\bar\Phi\big|\hat\mcF_{t}\big]+ \check\E^{\check\nu}\big[\ett_{E_K^c}\bar\Phi\big|\hat\mcF_{t}\big].
\end{align*}
Concerning the middle term, we have
\begin{align*}
  \check\E^{\check\nu^n}\big[\ett_{E_K^c}\bar\Phi\big|\hat\mcF_{t}\big]&= C\check\E^{\check\nu^n}\big[\ett_{E_K^c}(1+\|\check X^{t,\bx}\|^q_T)\big|\hat\mcF_{t}\big]
  \\
  &\leq \frac{C}{K}\check\E^{\check\nu^n}\big[\|\check X^{t,\bx}\|_T(1+\|\check X^{t,\bx}\|^q_T)\big|\hat\mcF_{t}\big]
  \\
  &\leq \frac{C}{K}(1+\|\bx\|^{q+1}_{t})
\end{align*}
and similarly for the last term, where $C>0$ does not depend on $K$ or $n$. On the other hand, by dominated convergence we have
\begin{align*}
\int_{0}^T\int_U(1-\check\nu^n_r(e))\lambda(de)dr \to \int_{0}^T\int_U(1-\check\nu_r(e))\lambda(de)dr,
\end{align*}
$\check\Prob$-a.s.~and
\begin{align*}
  \prod_{\hat\sigma_j\leq s}\check\nu^n_{\hat\sigma_j}(\hat\zeta_j)\to \prod_{\hat\sigma_j\leq s}\check\nu_{\hat\sigma_j}(\hat\zeta_j),
\end{align*}
$\check\Prob$-a.s.~as the number of terms in the product is $\check\Prob$-a.s.~finite effectively implying that $\hat\kappa^{\check\nu^n}_T\to \hat\kappa^{\check\nu}_T$, $\check\Prob$-a.s. Concerning the first term, we thus have
\begin{align*}
\check\E\big[\ett_{E_K}|\hat\kappa^{\check\nu^n}_{T}-\hat\kappa^{\check\nu}_T|\bar\Phi\big|\hat\mcF_{t}\big] &\leq C\check\E\big[|\hat\kappa^{\check\nu^n}_T-\hat\kappa^{\check\nu}_T|(1+K^{q})\big|\hat\mcF_{t}\big],
\end{align*}
where the latter tends to 0, $\hat\Prob$-a.s., as $n\to\infty$ by dominated convergence and the fact that $\hat\kappa^{\check\nu^n}_T\to\hat\kappa^{\check\nu}_T$, $\hat\Prob$-a.s. For each $K>0$, there is thus a $\check\Prob$-null set $E\subset\hat\mcF$ such that
\begin{align*}
  \lim_{n\to\infty}\check\E\big[|\hat\kappa^{\check\nu^n}_T-\hat\kappa^{\check\nu}_T|\bar\Phi\big|\hat\mcF_{t}\big]&\leq \frac{C}{K}(1+\|\bx\|^{q+1}_{t})
\end{align*}
on $\hat\Omega\setminus E$. Since $K>0$ was arbitrary, we conclude that the left-hand side equals 0, $\check\Prob$-a.s.
\end{proof}

\medskip

\begin{proof}[Proof of \eqref{ekv:new-duality}]
That $v^\mcR\geq v$ is evident from Theorem~\ref{thm:zs-game} and Lemma~\ref{lem:equiv}. By the definition of $u^\varrho$ and Lemma~\ref{lem:hat-zeta-approx}, for all sufficiently small $\zeta>0$ there exists $\eps>0$ such that
\begin{align*}
  v^{F,*}(t,\bx)\geq \hat J^{F,*}(t,\bx;u^\varrho)-\varrho \geq \hat J^{F,\zeta}(t,\bx;\check u^{\varrho,\eps})-2\varrho.
\end{align*}
Moreover,
\begin{align*}
  \hat J^\zeta(t,\bx;\check u^{\varrho,\eps})=\check J^{\mcR,F,\zeta}(t,\bx;\check \nu^{\varrho,\eps}).
\end{align*}
Furthermore, by Lemma~\ref{lem:check-nu-approx}
\begin{align*}
  \check J^{\mcR,F,\zeta}(t,\bx;\check \nu^{\varrho,\eps})
  = \lim_{n\to\infty}\check J^{\mcR,F,\zeta}(t,\bx;\check\nu^{\varrho,\eps,n})
  \geq v^{\mcR,F,\zeta}(t,\bx).
\end{align*}
Combining the above estimates yields
\begin{align*}
  v^{F,*}(t,\bx)\geq v^{\mcR,F,\zeta}(t,\bx)-2\varrho.
\end{align*}
Letting $\zeta\searrow 0$ and using that $\varrho>0$ was arbitrary yields $v^{F,*}(t,\bx)\geq v^{\mcR}(t,\bx)$ and thus $v^{F,*}(t,\bx)= v(t,\bx)$ for all $(t,\bx)\in [0,T]\times\bfC^d$.
\end{proof}

\section{Application to nonzero-sum controller-stopper games\label{sec:nz}}

In this section, we consider the nonzero-sum game between a controller and a stopper. The objective of the controller is now to maximize a different payoff
\begin{align*}
  J^C(t,\bx;u,\tau^F(X^u))&:=\E\Big[\psi^C(\tau^F(X^{t,\bx;u}),X^u)+\int_0^{\tau^F(X^{t,\bx;u})}f^C(t,X^{t,\bx;u},u_t)dt\Big].
\end{align*}
The stopper, on the other hand, selects a feedback stopping strategy $\tau^F \in \mathcal{T}^F$, that maximizes her payoff given the controller's action,
\begin{align*}
  J^S(t,\bx;u,\tau^F(X^u))&:=\E\Big[\psi^S(\tau^F(X^{t,\bx;u}),X^{t,\bx;u})+\int_0^{\tau^F(X^{t,\bx;u})}f^S(t,X^{t,\bx;u},u_t)dt\Big].
\end{align*}
Assuming that the coefficients $(f^C,\psi^C)$ and $(f^S,\psi^S)$ satisfy the conditions on $f,\psi$ in Assumption~\ref{ass:oncoeff}, we now apply the results of the previous sections to show that the feedback stopping rule
\begin{align*}
  \tau^{F,*}_t:\bfC^d\to [t,T],\quad
  \bx\mapsto\inf\{s\geq t: v^S(s,\bx)=\psi^S(s,\bx)\}\in\mcT^F_t,
\end{align*}
constructed from the zero-sum game associated with the stopper’s reward functional,
\begin{align}\label{ekv:Nash-zs}
  v^S(t,\bx)
  := \esssup_{\tau^F\in\mcT^F_t}\essinf_{u\in\mcU_t}J^S(t,\bx;u,\tau^F(X^{t,\bx;u}))
  = \essinf_{u\in\mcU_t}\esssup_{\tau\in\mcT^F_t}J^S(t,\bx;u,\tau^F(X^{t,\bx;u})),
\end{align}
naturally induces approximate equilibria in the nonzero-sum setting.

More precisely, we prove that for any $\varepsilon>0$, there exists a control $u^\eps\in\mcU_t$ such that the pair $(u^\eps,\tau^{F,*})$ constitutes an $\varepsilon$-Nash equilibrium. This provides a canonical candidate for the stopper’s strategy, even though an exact Nash equilibrium in pure strategies need not exist in general. This finding is the main result of the present section and is summarized in the following proposition:
\begin{prop}\label{prop:eps-Nash}
For each $(t,\bx)\in [0,T]\times\bfC^d$ and $\eps>0$, there exists a $u^\eps\in\mcU_t$ such that
\begin{align*}
  \begin{cases}
    J^C(t,\bx;u^\eps,\tau^{F,*}_t(X^{t,\bx;u^\eps}))\geq J^C(t,\bx;u,\tau^{F,*}_t(X^{t,\bx;u}))-\eps\\
    J^S(t,\bx;u^\eps,\tau^{F,*}_t(X^{t,\bx;u^\eps}))\geq J^S(t,\bx;u^\eps,\tau^F(X^{t,\bx;u^\eps}))-\eps
  \end{cases}
\end{align*}
for all $(u,\tau^F)\in\mcU_t\times\mcT^F_t$.
\end{prop}

\noindent\emph{Proof.} By the definition of the essential supremum, there exists a $u^{\eps,1}\in\mcU_t$ such that
\begin{align}\label{ekv:Nash-cont-opt}
  J^C(t,\bx;u^{\eps,1},\tau^{F,*}_t(X^{t,\bx;u^{\eps,1}}))\geq J^C(t,\bx; u,\tau^{F,*}_t(X^{t,\bx; u}))-\eps
\end{align}
for all $u\in\mcU_t$. Note that the left-hand side of the above equation does not depend on the values that $u^{\eps,1}$ takes on $[\tau^{F,*}_t(X^{t,\bx;u^{\eps,1}}),T]$. This allows us to manipulate the control in order to ensure that stopping at $\tau^{F,*}_t(X^{t,\bx;u^{\eps,1}})$ is an $\eps$-optimal decision for the stopper.

To save notation we extend the definition of the concatenation operator by letting, for each $\tau^F\in\mcT^F_t$ and $u^{1},u^2\in\mcU_t$,
\begin{align*}
  (u^{1}\otimes_{\tau^{F}}u^2)_s:=u^{1}_s\ett_{[0, \tau^{F}(X^{t,\bx;u^{1}})]}(s)+u^2_s\ett_{(\tau^{F}(X^{t,\bx;u^{1}}),T]}(s).
\end{align*}
By standard arguments, the control $u^{1}\otimes_{\tau^{F}}u^{2}$ is $\Prog(\bbF)$-measurable and therefore belongs to $\mcU_t$.

We divide the remainder of the proof into two steps:\\

\noindent\emph{Step 1.} We show that there is a $u^{2,\eps}\in\mcU_{\tau^{F,*}_t(X^{t,\bx;u^{1,\eps}})}$ such that $u^\eps:=u^{1,\eps}\otimes_{\tau^{F}}u^{2,\eps}$ satisfies
\begin{align}\label{ekv:tau-f-star-eps-opt-1}
  J^S(t,\bx;u^\eps,\tau\wedge\tau^{F,*}(X^{t,\bx;u^\eps}))\geq J^S(t,\bx;u^\eps,\tau)-\eps,\quad\forall \tau\in\mcT_{t}.
\end{align}
To accomplish this we first show that there is a $u^{2,\eps}\in\mcU_{\tau^{F,*}_t(X^{t,\bx;u^{1,\eps}})}$ such that
\begin{align}\label{ekv:u-2-cont-opt}
  J^S(\tau^{F,*}_t(X^{t,\bx;u^{1,\eps}}),X^{t,\bx;u^{1,\eps}};u^{2,\eps},\tau) \leq \psi^S(\tau^{F,*}_t(X^{t,\bx;u^{1,\eps}}),X^{t,\bx;u^{1,\eps}}) + \eps,
\end{align}
$\Prob$-a.s.~for any $\tau\in\mcT_{\tau^{F,*}_t(X^{t,\bx;u^{1,\eps}})}$. Using the notation $\tau^\eps:=\tau^{F,*}_t(X^{t,\bx;u^{1,\eps}})$, we find that since $v^S(\tau^\eps,X^{t,\bx;u^{1,\eps}})=\psi^S(\tau^{\eps},X^{t,\bx;u^{1,\eps}})$, \eqref{ekv:u-2-cont-opt} is equivalent to
\begin{align}\label{ekv:u-2-cont-opt-2}
  J^S(\tau^\eps,X^{t,\bx;u^{1,\eps}};u^{2,\eps},\tau)\leq v^S(\tau^\eps,X^{t,\bx;u^{1,\eps}}) + \eps,
\end{align}
$\Prob$-a.s.~for any $\tau\in\mcT_{\tau^{\eps}}$. Since $v^S$ and $f$ are continuous and of polynomial growth, while $\psi$ is left upper semi-continuous in $t$ and continuous in $\bx$ and of polynomial growth, and $\bfC^d$ is separable, there exists a sequence of sets $(\mcO^\eps_i)_{i\in\bbN}$, with $\mcO^\eps_i\subset \Lambda^{t,\bx}$ and a corresponding sequence $(t_i,\bx_i)_{i\in\bbN}$, with $(t_i,\bx_i)\in\mcO^\eps_i$, such that $\tilde t\leq t_i$, $|v(t_i,\bx_i)-v(\tilde t,\tilde\bx)|\leq \eps/3$ and
\begin{align}\label{ekv:J-app-in-mcO}
  \E\big[J^S(t_i,\bx_i;u,\tau\vee t_i)\,\big|\,\mcF_{\tilde t}\big]\leq J^S(\tilde t,\tilde \bx;u,\tau \vee \tilde t)+\eps/3,\quad\Prob\text{-a.s.}
\end{align}
for all $(\tilde t,\tilde\bx)\in\mcO^\eps_i$ and $(\tau,u)\in\mcT_t\times\mcU_t$. To obtain the latter inequality on the sets $\mcO^\eps_i$, we follow along the lines of the proof of Lemma 4.4 in \cite{cont-stop-P1}.

Now, for each $i\in\bbN$, there is a $u^{2,\eps,i}\in\mcU_{t}$ such that
\begin{align*}
  v(t_i,\bx_i)&\geq \esssup_{\tau\in\mcT_{t_i}}J^S(t_i,\bx_i;u^{2,\eps,i},\tau)-\eps/3.
\end{align*}
Consequently,
\begin{align*}
  \E\big[v^S(\tau^\eps,X^{t,\bx;u^{1,\eps}})\,\big|\,\mcF_{t}\big] & \geq \E\Big[\sum_{i\in\bbN}\ett_{\mcO^\eps_i}(\tau^\eps,X^{t,\bx;u^{1,\eps}}) v(t_i,\bx_i)\,\Big|\,\mcF_t\Big]-\eps/3
  \\
  & \geq \E\Big[\sum_{i\in\bbN}\ett_{\mcO^\eps_i}(\tau^\eps,X^{t,\bx;u^{1,\eps}}) J^S(t_i,\bx_i;u^{2,\eps,i},\tau)\,\Big|\,\mcF_t\Big]-2\eps/3
  \\
  & \geq \E\Big[\sum_{i\in\bbN}\ett_{\mcO^\eps_i}(\tau^\eps,X^{t,\bx;u^{1,\eps}}) J^S(\tau^\eps,X^{t,\bx;u^{1,\eps}};u^{2,\eps,i},\tau)\,\Big|\,\mcF_t\Big]-\eps,
\end{align*}
$\Prob$-a.s.~for any $\tau\in\mcT_{\tau^{\eps}}$. On the other hand, the control
\begin{align*}
  u^{2,\eps}:=\sum_{i\in\bbN}\ett_{\mcO^\eps_i}(\tau^\eps,X^{t,\bx;u^{1,\eps}}) u^{2,\eps,i}
\end{align*}
is $\Prog(\bbF)$-measurable and therefore belongs to $\mcU_{\tau^\eps}$, while satisfying \eqref{ekv:u-2-cont-opt-2}. Letting $u^\eps:=u^{\eps,1}\otimes_{\tau^{F,*}_t}u^{2,\eps}$ we thus have that
\begin{align*}
  &J^S(t,\bx;u^\eps,\tau)=\E\Big[\psi^S(\tau,X^{t,\bx;u^\eps})+\int_t^{\tau}f^S(s,X^{t,\bx;u^\eps},u^\eps_s)ds\,\Big|\,\mcF_t\Big]
  \\
  &=\E\Big[\ett_{[\tau\leq \tau^{F,*}_t(X^{t,\bx;u^\eps})]}\Big(\psi^S(\tau,X^{t,\bx;u^\eps})+\int_t^{\tau}f^S(s,X^{t,\bx;u^\eps},u^\eps_s)ds\Big)
  \\
  &\quad +\ett_{[\tau > \tau^{F,*}_t(X^{t,\bx;u^\eps})]}\Big(\int_t^{\tau^{F,*}_t(X^{t,\bx;u^{\eps}})}f^S(s,X^{t,\bx;u^\eps},u^\eps_s)ds +J^S(\tau^{F,*}_t(X^{t,\bx;u^{\eps}}),X^{t,\bx;u^{\eps}};u^{2,\eps},\tau)\Big)\,\Big|\,\mcF_t\Big]
  \\
  &\leq\E\Big[\ett_{[\tau\leq \tau^{F,*}_t(X^{t,\bx;u^\eps})]}\Big(\psi^S(\tau,X^{t,\bx;u^\eps})+\int_t^{\tau}f^S(s,X^{t,\bx;u^\eps},u^\eps_s)ds\Big)
  \\
  &\quad+\ett_{[\tau > \tau^{F,*}_t(X^{t,\bx;u^\eps})]}\Big(\int_t^{\tau^{F,*}_t(X^{t,\bx;u^{\eps}})}f^S(s,X^{t,\bx;u^\eps},u^\eps_s)ds +\psi^S(\tau\wedge \tau^{F,*}_t(X^{t,\bx;u^\eps}),X^{t,\bx;u^\eps})+\eps\Big)\,\Big|\,\mcF_t\Big]
  \\
  &\leq \E\Big[\psi^S(\tau\wedge \tau^{F,*}_t(X^{t,\bx;u^\eps}),X^{t,\bx;u^\eps})+\int_t^{\tau\wedge \tau^{F,*}_t(X^{t,\bx;u^\eps})}f^S(s,X^{t,\bx;u^\eps},u^\eps_s)ds\,\Big|\,\mcF_t\Big]+\eps.
\end{align*}
In particular, it follows that \eqref{ekv:tau-f-star-eps-opt-1} holds.\\

\noindent\emph{Step 2.} We extend \eqref{ekv:tau-f-star-eps-opt-1} by proving that
\begin{align}\label{ekv:tau-f-star-eps-opt-2}
  J^S(t,\bx;u^\eps,\tau^{F,*}(X^{t,\bx;u^\eps}))\geq J^S(t,\bx;u^\eps,\tau)-\eps,\quad\forall \tau\in\mcT_{t}.
\end{align}
Since $v^S$ is the corresponding value function and immediately stopping is always an option, we have that $v^S\geq \psi^S$. We pick an arbitrary $\tau\in\mcT_t$ and thus have by \eqref{ekv:tau-f-star-eps-opt-1} that
\begin{align}\label{ekv:Nash-steg}
  J^S(t,\bx;u^\eps,\tau)&\leq \E\Big[v^S(\tau\wedge \tau^{F,*}_t(X^{t,\bx;u^\eps}),X^{t,\bx;u^\eps})+\int_t^{\tau\wedge \tau^{F,*}_t(X^{t,\bx;u^\eps})}f^S(s,X^{t,\bx;u^\eps},u^\eps_s)ds\,\Big|\,\mcF_t\Big]+\eps.
\end{align}
On the other hand, since $\tau^{F,*}_t$ is an optimal feedback stopping rule, we find by Theorem~\ref{thm:tau-F-robust} that
\begin{align*}
  v^S(\tau\wedge \tau^{F,*}_t(X^{t,\bx;u^\eps}),X^{t,\bx;u^\eps}) & \leq J^S(\tau\wedge \tau^{F,*}_t(X^{t,\bx;u^\eps}),X^{t,\bx;u^\eps};u^\eps,\tau^{F,*}_t(X^{t,\bx;u^\eps})).
\end{align*}
Substituting this into \eqref{ekv:Nash-steg} gives
\begin{align}\nonumber
  J^S(t,\bx;u^\eps,\tau)&\leq \E\Big[J^S(\tau\wedge \tau^{F,*}_t(X^{t,\bx;u^\eps}),X^{t,\bx;u^\eps};u^\eps,\tau^{F,*}_t(X^{t,\bx;u^\eps}))
  \\
  &\quad+\int_t^{\tau\wedge \tau^{F,*}_t(X^{t,\bx;u^\eps})}f^S(s,X^{t,\bx;u^\eps},u^\eps_s)ds\,\Big|\,\mcF_t\Big]+\eps\nonumber
  \\
  &= \E\Big[\psi^S(\tau^{F,*}_t(X^{t,\bx;u^\eps}),X^{t,\bx;u^\eps}) + \int_t^{\tau^{F,*}_t(X^{t,\bx;u^\eps})}f^S(s,X^{t,\bx;u^\eps},u^\eps_s)ds\,\Big|\,\mcF_t\Big]+\eps\nonumber
  \\
  &=J^S(t,\bx;u^\eps,\tau^{F,*}_t(X^{t,\bx;u^\eps}))+\eps.\label{ekv:Nash-stop-opt}
\end{align}
Combining inequalities \eqref{ekv:Nash-cont-opt} and \eqref{ekv:Nash-stop-opt} establishes that the pair $(u^\eps,\tau^{F,*}_t)$ is an $\eps$-Nash equilibrium.\qed\\

The above proposition shows that $\tau^{F,*}$ appears in an $\varepsilon$-Nash equilibrium for any $\varepsilon>0$. It is therefore natural to regard $\tau^{F,*}$ as a plausible component of a pure Nash equilibrium. However, the supremum in
\begin{align}\label{ekv:controller-problem}
  \esssup_{u\in\mcU_t} J^C(t,\bx;u,\tau^{F,*}_t(X^{t,\bx;u}))
\end{align}
need not be attained by any admissible (strict) control $u^*\in\mcU_t$.

Indeed, if $(u^n)_{n\in\mathbb N}$ is a maximizing sequence for \eqref{ekv:controller-problem}, one typically only obtains convergence in law of the associated state processes (possibly along a subsequence), and only under restrictive conditions on the coefficients $a$, $\sigma$, $\psi^C$ and $f^C$ (see \eg \cite{FlemmingSoner2006}). Moreover, since $\tau^{F,*}$ is the first exit time from an open domain, convergence of the state process generally yields only lower semi-continuity of the form
\begin{align*}
  \liminf_{n\to\infty} \tau^{F,*}_t(X^{t,\bx;u^n})
  \;\ge\;
  \tau^{F,*}_t(X^{t,\bx;u}),
\end{align*}
which is insufficient to pass optimality to the limit.

To ensure that an optimal strict control exists and attains the supremum in \eqref{ekv:controller-problem}, additional structural assumptions are therefore required. In particular, conditions such as $f^C\le 0$ and monotonicity of the terminal reward, in this case that the map $t\mapsto \psi^C(t,\bx)$ is non-increasing for each $\bx\in\bfC^d$, are natural sufficient conditions ensuring optimality of the limiting control process.

\subsection{Nash equilibria and optimal control problems with discretionary stopping}
The $\varepsilon$-Nash equilibrium constructed above exhibits a pessimistic structure from the perspective of the stopper. Indeed, the value function $v^S$ is defined via a worst-case optimization over the controller’s actions, so that the feedback stopping rule $\tau^{F,*}$ is optimal under the assumption that the controller acts so as to minimize the stopper’s payoff. In this sense, the stopper behaves as if the controller were adversarial, even though the underlying game is nonzero-sum.

A more satisfactory approach is obtained by using the same ideas to derive an alternative formulation of the nonzero-sum game as a combined control and stopping problem for the controller. More precisely, one may view the controller as selecting both a control $u\in\mcU_t$ and a stopping rule $\tau^F\in\mcT^F_t$, subject to the constraint that $\tau^F$ takes values in the set
\[
  \mcT^{A}_t
  :=
  \Big\{
    \tau^F\in\mcT^F_t:
    v^S\big(\tau^F(X^{t,\bx;u}),X^{t,\bx;u}\big)
    =
    \psi^S\big(\tau^F(X^{t,\bx;u}),X^{t,\bx;u}\big)
  \Big\}.
\]
That is, the controller may only induce stopping at times at which stopping is optimal in the corresponding zero-sum game. In particular, stopping at times strictly before $\tau^F$ cannot be optimal for the stopper, since $v^S>\psi^S$ outside of the stopping region.

Within this formulation, the feedback rule $\tau^{F,*}_t$ corresponds to the minimal element of $\mcT^{A}_t$. In contrast, implementing other elements of $\mcT^{A}_t$ may be beneficial to both the controller and the stopper, unless $f^C\leq 0$ and the map $s\mapsto \psi^C(s,\tilde\bx)$ is non-increasing on $[t,T]$ for any $\tilde\bx\in\bfC^d$ satisfying $\tilde\bx\ett_{[0,t]}=\bx\ett_{[0,t]}$.

However, not every element of $\mcT^{A}_t$ can arise as part of an $\eps$-Nash equilibrium. To make this precise, let $(u,\tau^F)\in\mcU_t\times\mcT^F_t$ and define the $\bbF$-adapted, \cadlag process $(Y^{S,u,\tau^F}_s)_{s\in[t,T]}$ by
\begin{align*}
  Y^{S,u,\tau^F}_s
  =
  \E\Big[
    \psi^S\big(\tau^F(X^{t,\bx;u}),X^{t,\bx;u}\big)
    +
    \int_s^{\tau^F(X^{t,\bx;u})}
      f^S(r,X^{t,\bx;u},u_r)\,dr
    \,\Big|\,\mcF_s
  \Big],
  \quad \forall s\in [t,\tau^F(X^{t,\bx;u})].
\end{align*}
Then $u$ can be extended beyond $\tau^F(X^{t,\bx;u})$ in a way that $(u,\tau^F)$ forms an $\eps$-Nash equilibrium whenever
\begin{align*}
  \begin{cases}
    J^C(t,\bx;u,\tau^F(X^{t,\bx;u}))
    \geq
    J^C(t,\bx;\tilde u,\tau^F(X^{t,\bx;\tilde u})) - \eps,
    & \forall \tilde u\in\mcU_t, \\[0.5em]
    Y^{S,u,\tau^F}_s
    \geq
    \psi^S(s,X^{t,\bx;u}) - \eps,
    & \forall s\in [t,\tau^F(X^{t,\bx;u})].
  \end{cases}
\end{align*}

The second condition can be interpreted as a dynamic participation constraint for the stopper: at any time prior to $\tau^F(X^{t,\bx;u})$, deviating to immediate stopping does not improve her payoff by more than $\eps$. This effectively reduces the strategic interaction to a single-agent control problem with discretionary stopping under incentive compatibility constraints.

\bibliographystyle{plain}
\bibliography{cont-stop-P2_ref}

\begin{thebibliography}{10}

\bibitem{Bandini18}
E.~Bandini, A.~Cosso, M.~Fuhrman, and H.~Pham.
\newblock Backward sdes for optimal control of partially observed
  path-dependent stochastic systems: a control randomization approach.
\newblock {\em Ann. Appl. Probab.}, 28(3):1634--1678, 2018.

\bibitem{BayraktarHuang2013}
E.~Bayraktar and Y.-J. Huang.
\newblock On the multidimensional controller-and-stopper games.
\newblock {\em SIAM J. Control Optim}, 51(2):1263--1297, 2013.

\bibitem{BayraktarYao14}
E.~Bayraktar and S.~Yao.
\newblock On the robust optimal stopping problem.
\newblock {\em SIAM J. Control Optim.}, 52(5):3135--3175, 2014.

\bibitem{Bodnariu2024}
A.~Bodnariu and K.~Lindensj{\"o}.
\newblock A controller-stopper-game with hidden controller type.
\newblock {\em Stochastic Process. Appl.}, 173, 2024.

\bibitem{Choukroun15}
S.~Choukroun, A.~Cosso, and H.~Pham.
\newblock Reflected bsdes with nonpositive jumps, and controller-and-stopper
  games.
\newblock {\em Stochastic Process. Appl.}, 125:597--633, 2015.

\bibitem{EkstromSalami}
E.~Ekstr{\"o}m, K.~Lindensj{\"o}, and M.~Olofsson.
\newblock How to detect a salami slicer: A stochastic controller-and-stopper
  game with unknown competition.
\newblock {\em SIAM J. Control Optim}, 60(1):545--574, 2022.

\bibitem{EkstromDeFinetti}
E.~Ekstr{\"o}m, A.~Milazzo, and M.~Olofsson.
\newblock The de {Finetti} problem with uncertain competition.
\newblock {\em SIAM J. Control Optim}, 61(5):2997--3017, 2023.

\bibitem{FlemmingSoner2006}
W.~H. Fleming and H.~M. Soner.
\newblock {\em Controlled Markov Processes and Viscosity Solutions}.
\newblock Springer Science+Business Media, Inc., second edition, 2006.

\bibitem{Fuhrman2020}
M.~Fuhrman and M.~Morlais.
\newblock Optimal switching problems with an infinite set of modes: An approach
  by randomization and constrained backward sdes.
\newblock {\em Stochastic Process. Appl.}, 130:5(5):3120--3153, 2020.

\bibitem{Fuhrman15}
M.~Fuhrman and H.~Pham.
\newblock Randomized and backward sde representation for optimal control of
  non-markovian sdes.
\newblock {\em Ann. Appl. Probab.}, 25(4):2134--2167, 2015.

\bibitem{Karatzas2001}
I.~Karatzas and W.~D. Sudderth.
\newblock The controller-and-stopper game for a linear diffusion.
\newblock {\em Ann. Probab.}, 29(3):1111--1127, 2001.

\bibitem{Zamfirescu08}
I.~Karatzas and I.-M. Zamfirescu.
\newblock Martingale approach to stochastic differential games of control and
  stopping.
\newblock {\em Ann. Probab.}, 36(4):1495--1527, 2008.

\bibitem{Kharroubi2010}
I.~Kharroubi, J.~Ma, H.~Pham, and J.~Zhang.
\newblock Backward sdes with constrained jumps and quasi-variational
  inequalities.
\newblock {\em Ann. Probab.}, 38(2):794--840, 2010.

\bibitem{NutzZhang15}
M.~Nutz and J.~Zhang.
\newblock Optimal stopping under adverse nonlinear expectation and related
  games.
\newblock {\em Ann. Appl. Probab.}, 25(5):2503--2534, 2015.

\bibitem{imp-stop-game}
M.~Perninge.
\newblock Optimal stopping of bsdes with constrained jumps and related zero-sum
  games.
\newblock {\em Stochastic Process. Appl.}, 173, 2024.

\bibitem{cont-stop-P1}
M.~Perninge.
\newblock A nonlinear snell envelope representation for path-dependent
  controller-stopper games.
\newblock {\em arXiv:2606.10494}, 2026.

\bibitem{QuenSul14}
M-C. Quenez and A.~Sulem.
\newblock Reflected bsdes and robust optimal stopping for dynamic risk measures
  with jumps.
\newblock {\em Stochastic Process. Appl.}, 124:3031--3054, 2014.

\bibitem{Weerasinghe2006}
A.~Weerasinghe.
\newblock A controller and a stopper game with degenerate variance control.
\newblock {\em Elect. Comm. in Probab.}, 11:89--99, 2006.

\end{thebibliography}
\end{document}